\documentclass[a4paper,11pt,reqno]{amsart}

\usepackage[a4paper,top=2.5cm,bottom=1.9cm,left=2.5cm,right=2.5cm]{geometry}

\usepackage{mathtools}
\usepackage{enumerate}
\usepackage{array}
\usepackage{mathrsfs}
\usepackage{amssymb,amsmath}
\usepackage[english]{babel}

\usepackage[bookmarks,colorlinks]{hyperref}

\usepackage{subfig}
\usepackage{float}
\usepackage{tikz}

\newcommand{\HH}{\mathcal{H}}

\newcommand{\ZZ}{\mathbb{Z}}

\newcommand{\PP}{\mathbb{P}}

\newcommand{\e}{\mathrm{e}}

\newcommand{\gr}{\mathfrak{gr}}
\newcommand{\rk}{\textnormal{rk}}

\renewcommand{\AA}{\mathbb{A}}

\newcommand{\Proj}{\textnormal{Proj}\,}

\newcommand{\xx}{\boldsymbol{x}}
\newcommand{\yy}{\boldsymbol{y}}
\newcommand{\uu}{\boldsymbol{u}}
\newcommand{\vv}{\boldsymbol{v}}
\newcommand{\CC}{\mathbb{C}}
\renewcommand{\a}{\mathrm{a}}
\newcommand{\K}{\mathbb{C}}

\renewcommand{\geq}{\geqslant}
\renewcommand{\leq}{\leqslant}
\newcommand{\MLdeg}{\textnormal{MLdeg}\,}

\newtheorem{theorem}{Theorem}[section]
\newtheorem{corollary}[theorem]{Corollary}
\newtheorem{proposition}[theorem]{Proposition}
\newtheorem{lemma}[theorem]{Lemma}

\theoremstyle{definition}
\newtheorem{definition}[theorem]{Definition}
\newtheorem{example}[theorem]{Example}
\newtheorem{remark}[theorem]{Remark}

\begin{document}

\title{The maximum likelihood degree of Fermat hypersurfaces}

\author[D.~Agostini]{Daniele Agostini}
\address{Daniele Agostini \\ Institut f\"ur Mathematik\\ Humboldt Universit\"at zu Berlin\\ Berlin\\ Germany.}
\email{\href{mailto:daniele.agostini@math.hu-berlin.de}{daniele.agostini@math.hu-berlin.de}}

\author[D.~Alberelli]{Davide Alberelli}
\address{Davide Alberelli \\ University of Osnabrueck \\ Institut f\"ur Mathematik \\ Albrechtstr. 28a \\ 49076 Osnabr\"uck \\ Germany.}
\email{\href{mailto:dalberelli@uos.de}{dalberelli@uos.de}}

\author[F.~Grande]{Francesco Grande}
\address{Francesco Grande\\ Freie Universit{\"a}t Berlin \\ Institut f\"ur Mathematik\\ Arnimallee 2 \\ 14169 Berlin \\ Germany.}
\email{\href{mailto:fgrande@zedat.fu-berlin.de}{fgrande@zedat.fu-berlin.de}}

\author[P.~Lella]{Paolo Lella}
\address{Paolo Lella\\ Dipartimento di Matematica dell'Universit\`{a} degli Studi di Trento\\ 
         Via Sommarive 14\\ 38123 Povo (Trento)\\ Italy.}
\email{\href{mailto:paolo.lella@unitn.it}{paolo.lella@unitn.it}}
\urladdr{\url{http://www.paololella.it}}

\thanks{D.~Agostini and P.~Lella would like to thank CIME-CIRM for the financial support that allowed them to attend the school.
D. Agostini was supported by DFG within the research training group \lq\lq Moduli and Automorphic Forms\rq\rq\ (GRK1800). D.~Alberelli was supported by DFG within the research training group \lq\lq Kombinatorische Strukturen in der Geometrie\rq\rq\ (GK1916).
F.~Grande was supported by DFG within the research training group \lq\lq Methods for Discrete Structures\rq\rq\ (GRK1408). P.~Lella was partially supported by GNSAGA of INdAM, by PRIN 2010--2011 \lq\lq Geometria delle variet\`a algebriche\rq\rq\ and by FIRB 2012 \lq\lq Moduli spaces and Applications\rq\rq.
}

\begin{abstract}
We study the critical points of the likelihood function over the Fermat hypersurface. This problem is related to one of the main problems in statistical optimization: maximum likelihood estimation. The number of critical points over a projective variety is a topological invariant of the variety and is called maximum likelihood degree. We provide closed formulas for the maximum likelihood degree of any Fermat curve in the projective plane and of Fermat hypersurfaces of degree 2 in any projective space. Algorithmic methods to compute the ML degree of a generic Fermat hypersurface are developed throughout the paper. Such algorithms heavily exploit the symmetries of the varieties we are considering. A computational comparison of the different methods and a list of the maximum likelihood degrees of several Fermat hypersurfaces are available in the last section.
\end{abstract}


\maketitle

\section*{Introduction}

One of the main problems in statistics is the maximum likelihood estimation (MLE) of a statistical model. It has been widely explored and provided of very efficient methods. Recently a new branch of mathematics, called algebraic statistics, opened new horizons to the study of some statistical models by means of polynomial algebra. Among others, the result of uniqueness of the MLE for linear and log-linear (toric) models can be easily deduced using the algebraic approach \cite[Proposition 1.4, Theorem 1.10]{PachterSturmfels}.

However, most algebraic statistical models have more than one local maximum for the maximum likelihood estimation problem (see for instance \cite[Example 3.26]{PachterSturmfels}). Thus, it is natural to ask how many critical points of the likelihood function lie on the model; this question is known as the maximum likelihood (ML) degree problem. Some machinery from algebraic geometry has been applied to study particular cases of the ML degree problem; such synergy has produced quite a few results (see for instance \cite{HKS,BHR,HS,DraismaRodriguez,GDP,Uhler,HRS}). In particular, many hidden geometric features of algebraic statistical models have been brought to light. Most results in this area have been gathered and extended in \cite{HuhSturmfels}.

The task is to compute the ML degree of a Fermat hypersurface of degree $d$ in the complex projective space $\PP^n$. No closed formula $\phi(n,d)$ for computing the ML degree of Fermat hypersurfaces seems to exist for general $n$ and $d$. Hence, we look for the answer by a direct approach with the aid of algebraic geometry software (\emph{Macaulay2} \cite{M2}). 

The Fermat hypersurface $F_{n,d}$ of degree $d$ in the projective space $\PP^n$ is the zero locus of the polynomial
\begin{equation}\label{eq:fermat}
f_{n,d} := x_0^d + x_1^d + \ldots + x_n^d.
\end{equation}
The ML degree of $F_{n,d}$ is the number of critical points of the likelihood function
\begin{equation*}
\ell_{\boldsymbol{u}} := \dfrac{x_0^{u_0}\cdots x_n^{u_n}}{(x_0+\cdots+x_n)^{u_0+\cdots+u_n}},\qquad \boldsymbol{u} = (u_0,\ldots,u_n) \in \ZZ^{n+1}_{>0}
\end{equation*}
on the hypersurface defined by \eqref{eq:fermat} for a general $\boldsymbol{u}$. The standard approach to the problem is to consider the logarithmic derivatives of $\ell_{\boldsymbol{u}}$ and apply the theorem of Lagrange multipliers (see for instance \cite[Chapter 1]{PachterSturmfels} and \cite[Section 1]{HuhSturmfels}). Hence, $\boldsymbol{p} = (p_0,\ldots,p_n)$ is a critical point for $\ell_{\boldsymbol{u}}$ restricted to $F_{n,d}$ if, and only if, $f_{n,d}(\boldsymbol{p}) = 0$ and the rank of the matrix
\begin{equation}\label{eq:matrixintro}
\left(
\begin{array}{cccc}
\dfrac{u_0}{p_0} & \dfrac{u_1}{p_1} & \cdots & \dfrac{u_n}{p_n} \\
\phantom{\bigg\vert}1\phantom{\bigg\vert} & 1 & \cdots & 1 \\
d p_{0}^{d-1} & d p_{1}^{d-1} & \cdots & d p_{n}^{d-1}
\end{array}
\right)
\end{equation}
is not maximal. It is clear after few computations that even for small values of $n$ and $d$ the time needed to get an answer via the previous definition of critical point becomes huge. Refining this approach, in this paper we succeed in computing the ML degree of $F_{n,d}$ for a larger set of $n$ and $d$ (see Table \ref{tab:MLdegree}). This requires an algorithmic method that uses tools of algebraic geometry and exploits the symmetries of the Fermat hypersurface. We do not explore the possible statistical applications. Nevertheless, it may happen that similar ideas can be applied to other classes of highly symmetric models derived from statistical observations.

\medskip

Section \ref{problempresent} of the paper presents a formulation of the maximum likelihood degree problem in the language of algebraic geometry and introduces the standard  concepts and tools required in the rest of the paper. Moreover, after pointing out the key difficulties of the algorithmic procedure, we illustrate a first try to improve the computations.

Section \ref{problemsymmetrization} contains the main result of the paper for the Fermat hypersurfaces. We show that the ML degree of $F_{n,d}$ can be computed by setting to $1$ all the entries of the data vector $\boldsymbol{u}$. This is the point where algebraic geometry plays a key role. Indeed, to prove this fact, we construct a family of schemes $\mathcal{X} \rightarrow \AA^1$ such that the fiber over $t\in\AA^1$ is the ideal which encodes the solution of the ML degree problem for a family of data vectors $\boldsymbol{u}_t$ (the vectors depend on $t$). We are particularly interested in the ideals defining two special fibers. The first one solves correctly the ML degree problem $F_{n,d}$ and the second one describes the critical points in the case $\boldsymbol{u}= (1,\ldots,1)$. We prove that both ideals have the same number of solutions by showing that the family is flat.

In Section \ref{solutionfinding}, we investigate the symmetries of the problem. In fact, using the data vector $(1,\ldots, 1)$, the action of the symmetric group $\mathcal{S}_{n+1}$ on $F_{n,d}$ extends to the matrix \eqref{eq:matrixintro} and, in particular, to the critical points of the likelihood function. 
By looking at the orbits of the critical points and studying the number of distinct coordinates a critical point might have, we are able to subdivide the ML degree computation  into parallel subtasks whose computations involve less than $n + 1$ variables and, for this reason, are easier. (See \cite{GR} for another example of ML degree computation based on subdividing the main problem into several simpler problems.)

Section \ref{closedformula} is dedicated to closed formulas for two special families of Fermat hypersurfaces, namely $F_{n,2}$ and $F_{2,d}$. The first formula is an application of the results achieved in Section \ref{solutionfinding}, while the second one is obtained using topological arguments and is amazingly simple, based exclusively on the congruence modulo $6$ of the degree $d$.

The last section reports the computational results by mean of comparison tables between the running times of the naive algorithms and the improved algorithms. It is immediate to see that the advantages of using the second one are remarkable when $d\ll n$. As a conclusion we include a table with the maximum likelihood degrees for the Fermat hypersurfaces that have been computed so far.

\section{The ML degree problem}\label{problempresent}

The maximum likelihood degree problem we are facing has been described in the introduction. The formulation in  the language of algebraic geometry is the following: we want to study the variety of $\PP^n$ defined by the $3\times3$ minors of
\begin{equation}\label{eq:matrix}
\left(\begin{array}{cccc}
u_0 & u_1 & \ldots & u_n \\
x_0 & x_1 & \ldots & x_n \\
x_{0}^{d} & x_{1}^{d} & \ldots & x_{n}^{d}
\end{array}
\right)
\end{equation}
and the equation \eqref{eq:fermat} (notice that the matrix \eqref{eq:matrix} is obtained from the matrix \eqref{eq:matrixintro} by multiplying the $i$-th column by $p_i$ for all $i=0,\ldots,n$). In addition, by similarity to some constraints needed in the classical statistical modelling, we only want to consider the points not lying on the hyperplanes defined by the equations $x_0 = \ldots = x_n = x_0 + \cdots + x_n = 0$. We call this hyperplane arrangement the \emph{distinguished arrangement} $\mathcal{H}$.
Throughout the paper, the notation $\mathcal{M}^{\boldsymbol{u}}_{n,d}$ will be used to refer to the matrix \eqref{eq:matrix}. We denote by $\e_{ijk}(\mathcal{M}^{\boldsymbol{u}}_{n,d})$ the minor of $\mathcal{M}^{\boldsymbol{u}}_{n,d}$ corresponding to the columns $i,j,k$, i.e.
\begin{equation*}
\e_{ijk}(\mathcal{M}^{\boldsymbol{u}}_{n,d}) = u_i \,x_jx_k(x_k^{d-1}-x_j^{d-1}) + u_j\, x_kx_i(x_i^{d-1}-x_k^{d-1}) + u_k \,x_ix_j(x_j^{d-1}-x_i^{d-1}),
\end{equation*}
and by $I_{n,d}^{\boldsymbol{u}}$ the ideal generated by the $3\times3$ minors and the Fermat equation:
\begin{equation*}
I_{n,d}^{\boldsymbol{u}} := \left( f_{n,d},\ \e_{ijk}(\mathcal{M}^{\boldsymbol{u}}_{n,d})\ \vert\ \forall\ 0 \leqslant i < j < k \leqslant n \right) \subset \K[x_0,\ldots,x_n].
\end{equation*}
The critical points we are looking for are the solutions of the ideal $I_{n,d}^{\boldsymbol{u}}$ not lying on the distinguished arrangement, i.e.~the ones defined by the saturated ideal
\begin{equation}\label{eq:ideal}
I_{n,d}^{\boldsymbol{u}}\setminus\mathcal{H} := \left( I_{n,d}^{\boldsymbol{u}} : \big(x_0 \cdots x_n (x_0 + \cdots + x_n)\big)^{\infty}\right).
\end{equation}
If we consider also the entries of the data vector $\boldsymbol{u}$ as coordinates of a point in another projective space $\PP^n = \Proj \K[y_0,\ldots,y_n]$, it is natural to consider the \emph{likelihood correspondence}, which is the universal family of these critical points, i.e.~the subscheme $\mathcal{L}_{n,d} \subset \PP^n \times \PP^n$ defined as the closure of
\begin{equation*}
\left\{ (\boldsymbol{p},\uu)\in \PP^n \times \PP^n\ \vert\ \boldsymbol{p} \text{ is solution of } I_{n,d}^{\uu}\setminus \mathcal{H}\right\}
\end{equation*}
whose equations are described by the ideal $I^{\boldsymbol{y}}_{n,d}\setminus \mathcal{H} \subset \K[x_0,\ldots,x_n,y_0,\ldots,y_n]$.

Let us denote by $\PP^n_{\boldsymbol{x}}$ the projective space containing the Fermat hypersurface, by $\PP^n_{\yy}$ the projective space of data and by $\pi_{\xx}$ and $\pi_{\yy}$ the standard projections onto the two factors. Now we recall two theorems that motivate the definition of ML degree we will consider.

\begin{theorem}[{\cite[Theorem 1.6]{HuhSturmfels}}]
The likelihood correspondence $\mathcal{L}_X$ of any irreducible subvariety $X \subset \PP^n_{\boldsymbol{x}}$ is an irreducible variety of dimension $n$ in the product $\PP^n_{\boldsymbol{x}}\times\PP^n_{\boldsymbol{y}}$. The map $\pi_{\boldsymbol{y}} : \mathcal{L}_X \rightarrow\PP^n_{\boldsymbol{y}}$ is generically finite-to-one.
\end{theorem}

\begin{theorem}[{\cite[Theorem 1.15]{HuhSturmfels}}]\label{th:1.15HS}
Let $\uu \in \mathbb{R}^{n+1}_{>0}$, and let $X \subset \PP^n$ be an irreducible variety such that no singular 
point of any intersection $X \cap \{x_i = 0\}$ lies in the hyperplane $\{x_0+\cdots+x_n = 0\}$. Then
\begin{enumerate}[1.]
\item the likelihood function $\ell_{\uu}$ on $X$ has only finitely many critical points in $X_{\textnormal{reg}}\setminus \mathcal{H}$;
\item if the fiber $\pi_{\yy}^{-1}(\uu)$ is contained in $X_{\textnormal{reg}}$, then its length equals the ML degree of $X$.
\end{enumerate}
\end{theorem}

Notice that the hypotheses of Theorem \ref{th:1.15HS} are automatically verified in the case we are interested in. Indeed, the Fermat hypersurface $F_{n,d}$ is smooth, i.e.~$(F_{n,d})_{\textnormal{reg}} = F_{n,d}$, and every intersection $F_{n,d} \cap \{x_i = 0\}$ is again a Fermat hypersurface in a space of lower dimension. Thus, we adopt the following definition.

\begin{definition}
The \emph{ML degree $\MLdeg X$} of the irreducible subvariety $X\subset \PP^n_{\boldsymbol{x}}$ is the degree of the projection of the likelihood correspondence $\mathcal{L}_X$ to the second factor $\pi_{\boldsymbol{y}}: \mathcal{L}_X \rightarrow\PP^n_{\boldsymbol{y}}$.
\end{definition}

In order to determine the ML degree of the Fermat hypersurface, we can begin considering the following two standard approaches.

\smallskip

\noindent\textbf{Multidegree.} We can consider the multidegree of the likelihood correspondence $\mathcal{L}_{n,d}$ in the sense of \cite[Chapter 8]{MillerSturmfels} with respect to the natural $\ZZ^2$-grading on the polynomial ring $\K[x_0,\ldots,$ $x_n,y_0,\ldots,y_n]$. The multidegree of $\mathcal{L}_{n,d}$ is a polynomial $B_{\mathcal{L}_{n,d}}$ in the ring $\ZZ[T_{\boldsymbol{x}},T_{\boldsymbol{y}}]$ of degree $n = \dim \mathcal{L}_{n,d}$. It can be computed by means of the prime ideal defining $\mathcal{L}_{n,d}$ \cite[Proposition 8.49]{MillerSturmfels} and turns out to have the following shape 
\begin{equation}\label{eq:bidegree}
B_{\mathcal{L}_{n,d}}(T_{\boldsymbol{x}},T_{\boldsymbol{y}}) = (\MLdeg F_{n,d}) T_{\boldsymbol{x}}^n + \cdots + (\deg F_{n,d})T_{\boldsymbol{x}} T_{\boldsymbol{y}}^{n-1}.
\end{equation}
Hence, we can compute the ML degree of the Fermat hypersurface as the leading coefficient of the multidegree $B_{\mathcal{L}_{n,d}}(T_{\boldsymbol{x}},T_{\boldsymbol{y}})$ of the likelihood correspondence.

\smallskip

\noindent\textbf{Random data.} The degree of the map $\pi_{\boldsymbol{y}}: \mathcal{L}_{n,d} \rightarrow\PP^n_{\boldsymbol{y}}$ is the degree of the generic fiber of $\pi_{\boldsymbol{y}}$. Thus, there is an open dense subset $\mathcal{U} \subset\PP^n_{\boldsymbol{y}}$ whose points have fiber of constant degree. A computational strategy to determine $\MLdeg F_{n,d}$ is to randomly pick a point of $\PP^n_{\yy}$ and to calculate the degree of its fiber. Indeed, the probability to randomly choose a point in the Zariski closed subset $\PP^n_{\boldsymbol{y}}\setminus \mathcal{U}$ is negligible, so that we get almost surely the degree of the projection $\pi_{\yy}$, i.e.~the ML degree.  

Computational experiments (see Table \ref{tab:2345}\subref{tab:cf1} and Table \ref{tab:2345}\subref{tab:cf2} for details) show that these two methods for determining the ML degree of $F_{n,d}$ have a long execution time even for small values of $n$ and $d$. This is mainly due to the computation of a Gr\"obner basis of the ideal, which is needed to determine the degree (resp.~multidegree) of the ideal of the generic fiber (resp.~likelihood correspondence). A little more accurate analysis of the computational experiments reveals that the time required by the elimination of the critical points lying on the distinguished arrangement $\mathcal{H}$ is the hardest part of the process. Indeed, to accomplish this task, we need to saturate the ideal $I_{n,d}^{\boldsymbol{u}}$ (resp.~$I_{n,d}^{\boldsymbol{y}}$) by $n+2$ linear forms ($x_0$, \ldots, $x_n$ and $x_0 + \cdots + x_n$) and this operation may be computationally expensive. In the case of the Fermat hypersurface, we see that we can reduce to saturation by a single linear form.


\begin{lemma}\label{prop:davide}
Let $\uu$ be a data vector such that $u_i\neq 0,\ \forall\ i,$ and let $\boldsymbol{p} = [p_0:\ldots:p_n] \in \PP^n_{\boldsymbol{x}}$ be a solution of the ideal $I_{n,d}^{\uu}$. If for some $i$ the coordinate $p_i$ vanishes, then $\sum_{j=0}^n p_j = 0$.
\end{lemma}
\begin{proof}
First, note that there are at least two non-zero coordinates, as the point is a solution of the equation $f_{n,d} = 0$. Let $p_k$ and $p_h$ be any two non-zero coordinates. Evaluating the minor $\e_{ikh}(\mathcal{M}^{\boldsymbol{u}}_{n,d})$ on the $3$-tuple $(p_i=0,p_k,p_h)$, we obtain 
\begin{equation*}
\e_{ikh}\big(\mathcal{M}^{\boldsymbol{u}}_{n,d}\big)(0,p_k,p_h) = u_i\, p_kp_h (p_h^{d-1}-p_k^{d-1}),
\end{equation*}
so that for any solution of the ideal we have $p_h^{d-1} = p_k^{d-1}$.
Finally,
\begin{equation*}
0 = \sum_{j=0}^n p_j^d = \sum_{\begin{subarray}{c}j=0\\ p_j \neq 0\end{subarray}}^n p_j^d = \sum_{\begin{subarray}{c}j=0\\ p_j \neq 0\end{subarray}}^n p_j\, p_j^{d-1} = p_k^{d-1}\sum_{\begin{subarray}{c}j=0\\ p_j \neq 0\end{subarray}}^n p_j = p_k^{d-1}\sum_{j=0}^n p_j,
\end{equation*}
where $p_k$ is just one of the non-zero coordinates.
\end{proof}

The previous lemma implies that saturating the ideal by the linear form $x_0+\cdots+x_n$ suffices to guarantee that no solution lies on the hyperplane arrangement $\mathcal{H}$ for a generic data vector $\boldsymbol{u}$. Saturating with respect to a unique linear form instead of $n+2$ is certainly an improvement. Unfortunately, the saturation with respect to $x_0+\cdots+x_n$ is much more expensive than the saturation with respect to a single variable. Next proposition shows that we can avoid the saturation by $x_0+\cdots+x_n$ and compute $\MLdeg F_{n,d}$ as the difference between the number of points defined by $I_{n,d}^{\uu}$ and the number of points defined by the ideal $I_{n,d}^{\uu} + (x_0+\cdots+x_n)$. 

\begin{proposition}\label{prop:daniele_smooth}
For a generic $\uu\in \mathbb{P}^n_{\yy}$ the points defined by $I^{\uu}_{n,d}$ lying on the hyperplane $x_0 + \ldots + x_n = 0$ are simple. Hence,
\begin{equation}\label{eq:deg_diff}
\deg \left(I_{n,d}^{\boldsymbol{u}}\setminus\mathcal{H}\right) = \deg I_{n,d}^{\boldsymbol{u}} - \deg \left(I_{n,d}^{\uu} + (x_0+\cdots+x_n)\right).
\end{equation}
\end{proposition}
\begin{proof}
By symmetry, we can restrict our attention to the affine chart $\mathcal{U}_0 = \{ x_0\ne 0 \}\subset\mathbb{P}^n_{\xx}$ with affine coordinates $t_i=\frac{x_i}{x_0}$ for $i=1,\dots,n$. The ideal $I^{\uu}_{n,d}$ restricted to $\mathcal{U}_0$ is defined by the polynomial $1 + t_1^d \dots + t_n^d$ and by the $3 \times 3$ minors of the matrix
\[ 
\left(\begin{array}{cccc}
u_0 & u_1 & \ldots & u_n \\
1 & t_1 & \ldots & t_n \\
1 & t_{1}^{d} & \ldots & t_{n}^{d}.
\end{array}\right).
\]
First, we observe that
\[
\rk \left(\begin{array}{cccc}
u_0 & u_1 & \ldots & u_n \\
1 & t_1 & \ldots & t_n \\
1 & t_{1}^{d} & \ldots & t_{n}^{d}.
\end{array}\right) = 
\rk \left(\begin{array}{cccc}
u_0+\cdots+u_n & u_1 & \ldots & u_n \\
1+t_1 + \cdots + t_n & t_1 & \ldots & t_n \\
1+t_1^d + \cdots + t_n^d & t_{1}^{d} & \ldots & t_{n}^{d}.
\end{array}\right).
\]
Since $\uu$ is generic we can assume that $\sum_{i=0}^n u_i \neq 0$ and then normalize to $\sum_{i=0}^n u_i = 1$. Moreover, $1+t_1^d + \cdots + t_n^d$ vanishes since the points we are looking at lie on the Fermat hypersurface $F_{n,d}$. Let $h=1+t_1+\dots+t_n$. Applying a sequence of column operations to the matrix yields the following equalities
\[
\begin{split}
\rk \left(\begin{array}{cccc}
1 & u_1 & \ldots & u_n \\
h & t_1 & \ldots & t_n \\
0 & t_{1}^{d} & \ldots & t_{n}^{d}
\end{array}\right) &{}= 
\rk \left(\begin{array}{cccc}
1 & 0 & \ldots & 0 \\
h & t_1 - u_1 h & \ldots & t_n-u_n  h \\
0 & t_{1}^{d} & \ldots & t_{n}^{d}
\end{array}\right) ={}\\
&{} = \rk \left(\begin{array}{ccc}
t_1 - u_1 h & \ldots & t_n-u_n  h \\
t_{1}^{d} & \ldots & t_{n}^{d}
\end{array}\right) +1 = {}\\
&{}=\rk \left(\begin{array}{cccc}
\ldots & t_i - u_i h & \ldots & (h-1)-h(1-u_0) \\
\ldots & t_{i}^{d} & \ldots & t_1^d + \cdots + t_{n}^{d}
\end{array}\right) +1 = {}\\
&{} = \rk \left(\begin{array}{cccc}
\ldots & t_i - u_i h & \ldots & u_0h-1 \\
\ldots & t_{i}^{d} & \ldots  & -1
\end{array}\right) +1 = {}\\
&{} = \rk \left(\begin{array}{cccc}
\ldots & (u_0h-1)t_i^d + t_i - u_i h & \ldots & u_0h-1 \\
\ldots & 0 & \ldots & -1
\end{array}\right) +1 
\end{split}
\]
 so that
\[
I^{\uu}_{n,d}\big\vert_{\mathcal{U}_0} = \left(1+t_1^d + \cdots + t_n^d,(u_0h-1)t_{i}^d + t_{i} - u_{i} h\ \vert\ \forall\ 1 \leq i\leq n-1\right).
\]

In order to prove that the solutions of $I_{n,d}^{\uu}\big\vert_{\mathcal{U}_0}$ that satisfy $h=0$ are simple, we compute the Jacobian matrix of $I_{n,d}^{\uu}\big\vert_{\mathcal{U}_0}$ and we check that the locus where it is not of maximal rank does not intersect the set defined by $I_{n,d}^{\uu}\big\vert_{\mathcal{U}_0} + (h)$. The Jacobian matrix restricted to the set $\{h = 0\}$ is
\[ 
\left(
\begin{array}{cccc}
dt_1^{d-1}  & \ldots & u_0t_i^d-u_i & \ldots  \\
\vdots &   & \vdots&    \\
 dt_i^{d-1} & \ldots  & u_0t_{i}^d-u_{i}+(1-dt_{i}^{d-1})  & \ldots \\
\vdots &  & \vdots&   \\
dt_n^{d-1}  & \ldots & u_0t_i^d-u_i & \ldots  \\
\end{array}
\right)
\]
and the rank of this matrix equals the rank of the matrix
\[ 
W:=
\left(
\begin{array}{cccccc}
t_1^{d-1} - t_n^{d-1} & 1 - dt_1^{d-1} &  \ldots & 0 & \ldots & 0 \\
\vdots & \vdots & \ddots  & \vdots&   & \vdots \\
 t_i^{d-1} - t_n^{d-1} & 0&\ldots  & 1-dt_{i}^{d-1}  & \ldots & 0\\
\vdots & \vdots & & \vdots & \ddots & \vdots   \\
t_{n-1}^{d-1} - t_n^{d-1} & 0 & \ldots & 0 & \ldots & 1 - dt_{n-1}^{d-1} \\
t_n^{d-1}  & u_0t_1^d-u_1 & \ldots & u_0t_i^d-u_i & \ldots & u_0t_{n-1}^d-u_{n-1} \\
\end{array}
\right).
\]
Now, we have to show that the ideal $J := I_{n,d}^{\uu}\big\vert_{\mathcal{U}_0} + (\det W, h)$ has no solution. Reducing by $h$ the generators of $I_{n,d}^{\uu}\big\vert_{\mathcal{U}_0}$ we obtain
\[ 
J = (\det W, h, 1+t_1^d+\dots+t_n^d, t_i^d-t_i\ \vert\ \forall\ 1 \leq i \leq n-1) = (\det W, h, t_i^d-t_i\ \vert\ \forall\ 1 \leq i \leq n), 
\]
which implies that every solution of $J$ satisfy the following condition: either $t_i = 0$ or $t^{d-1}_i = 1$ for every $i = 1,\ldots,n$. Since the matrix $W$ is symmetric in the variables $t_1,\ldots,t_{n-1}$, we can assume $t_i\ne 0$ for $1\leq i \leq r \leq n-1$ and $t_i=0$ for $r+1\leq i \leq n-1$. Thus, the matrix $W$ has the following form
\[ 
\left(\begin{array}{cccccccc}
1-t_{n}^{d-1}  &  1-d & \ldots & 0 & 0 & \ldots  & 0 \\
 \vdots & \vdots  & \ddots & \vdots & \vdots & 0 & \vdots \\
1-t_{n}^{d-1} & 0 & \ldots & 1-d & 0 & \ldots & 0\\ 
-t_{n}^{d-1} & 0 &  \ldots & 0   & 1 & \ldots & 0 \\
 \vdots & \vdots  &  0 & \vdots & \vdots & \ddots & \vdots \\
-t_{n}^{d-1} & 0  & \dots & 0 & 0 & \dots & 1 \\
t_{n}^{d-1} & u_0t_1 -u_1 & \dots & u_0t_{r}-u_{r} & -u_{r+1} & \dots &  -u_{n-1} \\
\end{array}\right)
\]
from which we can easily compute its determinant and express it (up to a constant) as follows
\[ 
\begin{split}
\det W &{}=t_n^{d-1} - \frac{1-t_n^{d-1}}{d-1}\sum_{i=1}^r (u_0t_i-u_i) - \sum_{i=r+1}^{n-1}u_it_n^{d-1} = {}\\&{} = \begin{cases}
\smallskip
\dfrac{1}{d-1}\left( \sum\limits_{i=1}^r u_i + u_0 \right),& \text{if }t_n =0,\\
1 - \sum\limits_{i=r+1}^{n-1} u_i,& \text{if } t_n^{d-1} = 1.
\end{cases}
\end{split}
\]
For a generic point $\uu \in \PP^n_{\yy}$, the determinant does not vanish and since this reasoning works for every $1 \leq r \leq n-1$, we proved that $J = (1)$. 

Finally, in order to prove the last statement, we notice that, as the points of $I_{n,d}^{\uu}$ lying on $\mathcal{H}$ in fact lies on $x_0 + \cdots + x_n = 0$ (Lemma \ref{prop:davide}) and are simple, the following equality holds:
\[
I_{n,d}^{\uu} \setminus \mathcal{H} =  \left( I_{n,d}^{\boldsymbol{u}} : (x_0 + \cdots + x_n)^{\infty}\right) = \left( I_{n,d}^{\boldsymbol{u}} : (x_0 + \cdots + x_n)\right).
\]
\end{proof}

\begin{remark}\label{rk:also1}
Notice that the proof of Proposition \ref{prop:daniele_smooth} applies also for $\uu=[1:\ldots:1]$.
\end{remark}

We conclude this section by describing the shape of the ideal $\big(I_{n,d}^{\uu} + (x_0+\cdots+x_n)\big)$.

\begin{lemma}\label{prop:notDependU}
Let $\uu=[u_0:\ldots:u_n]$ be a point in $\PP^n_{\yy}$ such that $u_0+\cdots+u_n \neq 0$. The solutions of $I_{n,d}^{\uu}$ lying on the hyperplane $x_0+\cdots+x_n=0$  do not depend on the point $\uu$. More precisely,
\begin{equation}\label{eq:pointRemoved}
I_{n,d}^{\uu} + (x_0+\cdots+x_n) = \left(\sum_{j=0}^n\, x_j,\ \sum_{j=0}^n\, x_j^d,\ x_k x_h(x_h^{d-1}-x_k^{d-1}),\ \forall\ 0\leqslant k < h \leqslant n\right).
\end{equation}
\end{lemma}
\begin{proof}
The ideal of $3\times 3$ minors of $\mathcal{M}_{n,d}^{\uu}$ contains also the minors of any matrix obtained from $\mathcal{M}_{n,d}^{\uu}$ by column operations. For instance, we can fix $i$ and replace the $i$-th column with the sum of all columns. We get the matrix
\small
\begin{equation*}
\left(\begin{array}{ccccccc}
u_0 & \ldots & u_{i-1}& \sum_j u_j & u_{i+1}& \ldots & u_n \\
x_0 & \ldots & x_{i-1}& \sum_j x_j & x_{i+1}& \ldots & x_n \\
x_0^d & \ldots & x_{i-1}^d & \sum_j x_j^d & x_{i+1}^d & \ldots & x_n^d \\
\end{array}
\right) =
\left(\begin{array}{ccccccc}
u_0 & \ldots & u_{i-1}& \sum_j u_j & u_{i+1}& \ldots & u_n \\
x_0 & \ldots & x_{i-1}& 0 & x_{i+1}& \ldots & x_n \\
x_0^d & \ldots & x_{i-1}^d & 0 & x_{i+1}^d & \ldots & x_n^d \\
\end{array}
\right)
\end{equation*}
\normalsize
as we are assuming both $x_0+\cdots+x_n$ and $x_0^d+\cdots+x_n^d$ equal to $0$. Among the minors of the second matrix, we have, up to sign, $(\sum_j u_j)x_k x_h (x_h^{d-1}-x_k^{d-1}),\ k,h \neq i$, and varying $i$, we prove that
\begin{equation*}
\left( \sum x_j,\ \sum x_j^d,\ x_kx_h(x_h^{d-1}-x_k^{d-1}),\ \forall\ 0\leqslant k < h \leqslant n\right) \subseteq I_{n,d}^{\uu} + (x_0+\cdots+x_n).
\end{equation*} 
The other inclusion is straightforward if we notice that the $3\times 3$ minors of $\mathcal{M}_{n,d}^{\uu}$ are linear combinations of the polynomials $x_kx_h(x_h^{d-1}-x_k^{d-1})$.
\end{proof}

See Table \ref{tab:2345}\subref{tab:cf1},\subref{tab:cf1plus} and Table \ref{tab:2345}\subref{tab:cf2},\subref{tab:cf2plus} for a comparison between the running time of the naive strategy (based on the saturation) and the strategy based on Proposition \ref{prop:daniele_smooth} applied both to the multidegree and random data approach.

\section{Symmetrizing the problem}\label{problemsymmetrization}

To improve further the computations, we would like to extend some symmetries of the Fermat hypersurfaces to symmetries of the ML degree problem. More precisely, we will consider the action of the symmetric group $\mathcal{S}_{n+1}$ on the variables of $\K[x_0,\ldots,x_n]$. By looking at the matrix \eqref{eq:matrix}, we notice that the ideal \eqref{eq:ideal} is symmetric with respect to $\mathcal{S}_{n+1}$ if the polynomials $\mathrm{e}_{ijk}(\mathcal{M}^{\boldsymbol{u}}_{n,d})$ are invariant under the action of $\mathcal{S}_{n+1}$. This is equivalent to the requirement that all the entries of the data vector $\boldsymbol{u}$ are equal. From a statistical point of view, we are restricting to the very specific case where we observe the same number of occurrences for each random variable. From the algebraic geometry point of view, we are claiming that the point $\boldsymbol{1} := [1:\ldots:1] \in \PP^n_{\boldsymbol{y}}$ belongs to the open subset $\mathcal{U}$ of points whose fiber has the correct degree.

For any multi-index $\boldsymbol{u} = (u_0,\ldots,u_n)$, let us denote by $\vert\boldsymbol{u}\vert$ the sum $u_0+\cdots+u_n$, by $\hat{u}_i$ the difference $\vert\uu\vert - u_i$ and by $\hat{\uu}$ the multi-index $(\hat{u}_0,\ldots,\hat{u}_n)$.  

Consider a generic point $\boldsymbol{u} \in \mathcal{U} \subset \PP^n_{\boldsymbol{y}}$ for the morphism $\mathcal{L}_{n,d} \rightarrow \PP^n_{\boldsymbol{y}}$, i.e.~such that the degree of its fiber equals the ML degree of $F_{n,d}$. We assume that $\vert\boldsymbol{u}\vert \neq 0$ and $u_i\neq 0$, for all $i$. We prove that the fiber of the point $[\vert\uu\vert:\ldots:\vert\uu\vert] = \boldsymbol{1} \in \PP^n_{\yy}$ has the same degree.
Consider the affine line in $\PP^n_{\yy}$ passing through $\boldsymbol{u}$ and $\boldsymbol{1}$, which is the image of the map $\phi:\AA^1 \rightarrow \PP^n_{\yy}$ induced by the ring homomorphism
\begin{equation}\label{eq:flatRing}
\begin{split}
\K[y_0,\ldots,y_n] &\longrightarrow\qquad \K[t]\\
\parbox{1.5cm}{\centering $y_i$} & \longmapsto u_it + \vert\uu\vert(1-t).
\end{split}
\end{equation}
We will show that the induced subfamily of the family $\mathcal{X}:=\Proj \big(\K[y_0,\ldots,y_n]/I_{n,d}^{\boldsymbol{y}}\big) \to \PP^n_{\boldsymbol{y}}$
\begin{center}
\begin{tikzpicture}
\node (A) at (0,0) [inner sep=2pt] {$\AA^1$};
\node (B) at (2.5,0) [inner sep=2pt] {$\PP^n_{\boldsymbol{y}}$};

\draw [->] (A) --node[above]{\small $\phi$} (B);

\node (C) at (0,1.5) [inner sep=2pt] {\small $\AA^1 \times_{\PP^n_{\boldsymbol{y}}} \mathcal{X}$};
\node (D) at (2.5,1.5) [inner sep=2pt] {$\mathcal{X}$};

\draw [->] (C) -- (D);

\draw [->] (C) -- (A);
\draw [->] (D) -- (B);
\end{tikzpicture}
\end{center}
is flat, i.e.~all the fibers have the same degree (see \cite[III, Theorem 9.9]{Hartshorne}).

To prove such a property, we will need a flatness criterion for filtered modules. Thus, we briefly recall few features of filtered modules (see \cite[Chapter 5]{Eisenbud}) for the particular case we are dealing with. Let us consider a polynomial ring $R$ and its irrelevant ideal $\mathfrak{m}$. The $\mathfrak{m}$-adic filtration of $R$ is the descending multiplicative filtration of ideals
\begin{equation*}
R \supset \mathfrak{m} \supset \mathfrak{m}^2 \supset \cdots \supset \mathfrak{m}^\ell \supset \cdots
\end{equation*}
that induces the standard graded structure of $R$ by considering the direct sum
\begin{equation*}
\mathfrak{gr}\, R := \bigoplus_{\ell\geqslant 0} \mathfrak{m}^\ell/\mathfrak{m}^{\ell+1} = \bigoplus_{\ell \geqslant 0} R_\ell.
\end{equation*}
The same construction extends to any $R$-module $M$. The $\mathfrak{m}$-adic filtration of $M$ is
\begin{equation*}
M \supset \mathfrak{m}M \supset \mathfrak{m}^2M \supset \cdots \supset \mathfrak{m}^\ell M \supset \cdots
\end{equation*}
and the associated graded $\gr\, R$-module is
\begin{equation*}
\gr\, M := \bigoplus_{\ell\geqslant 0} \mathfrak{m}^\ell M/\mathfrak{m}^{\ell+1}M = M/\mathfrak{m}M \oplus \mathfrak{m}M/\mathfrak{m}^2M \oplus \cdots
\end{equation*}
For any $f \in M$, we define the \emph{initial form of $f$} to be the element
\begin{equation*}
\text{in}(f) := f \bmod \mathfrak{m}^{\ell+1}M \subset \mathfrak{m}^{\ell}M/\mathfrak{m}^{\ell+1}M
\end{equation*}
where $\ell$ is the greatest index such that $f \in \mathfrak{m}^{\ell}M$. If $M'$ is a submodule of $M$, we can consider the $(\gr\, R)$-submodule of $\gr\, M$ generated by the elements $\text{in}(f),\ \forall\ f \in M'$. In particular, for any ideal $J \subset R$, we have $\gr(R/J) = (\gr\, R)/\text{in}(J) = R/\text{in}(J)$ (see \cite[Exercise 5.3]{Eisenbud}).

\begin{theorem}[{\cite[Theorem 2.20]{JJ}, \cite[Proposition 3.12]{Bjork}}]\label{th:flatCrit}
Let $S$ be a quotient of a polynomial ring and let $M$ be a $S$-module. If $\gr\, M$ is flat over $\gr\, S$, then $M$ is flat over $S$. 
\end{theorem}

With the help of these new tools, we can prove a crucial result for this paper.

\begin{lemma}\label{lem:flatness}
The family $\AA^1 \times_{\PP^n_{\boldsymbol{y}}} \mathcal{X} \rightarrow \AA^1$ induced by the morphism \eqref{eq:flatRing} is flat.
\end{lemma}
\begin{proof}
The family induced by \eqref{eq:flatRing} is described by the ideal 
\begin{equation*}
I_{n,d}^{\uu,t} := (F_{n,d}, \e_{ijk}(\mathcal{M}_{n,d}^{\uu,t}),\ \forall\ 0\leqslant i < j < k \leqslant n)
\end{equation*}
(without saturation by $x_0+\cdots+x_n$), where
\begin{equation*}
\mathcal{M}_{n,d}^{\uu,t} := \left(\begin{array}{ccccc}
\vert\uu\vert - \hat{u}_0t & \ldots & \vert\uu\vert - \hat{u}_it & \ldots & \vert\uu\vert - \hat{u}_nt \\
x_0 & \ldots & x_i & \ldots & x_n \\
x_{0}^{d} & \ldots & x_{i}^{d} & \ldots & x_{n}^{d}
\end{array}\right).
\end{equation*} 
We show that the module $\K[x_0,\ldots,x_n,t]/I_{n,d}^{\uu,t}$ is flat over $\K[t]$ via Theorem \ref{th:flatCrit}, i.e.~proving that the module $\gr\left(\K[x_0,\ldots,x_n,t]/I_{n,d}^{\uu,t}\right)$ is flat over $\gr\, \K[t] = \K[t]$. Since we have
\begin{equation*}
 \gr\left(\K[x_0,\ldots,x_n,t]/I_{n,d}^{\uu,t}\right) = \gr\, \K[x_0,\ldots,x_n,t]/\text{in}\big(I_{n,d}^{\uu,t}\big) = \K[x_0,\ldots,x_n,t]/\text{in}\big(I_{n,d}^{\uu,t}\big),
\end{equation*}
we focus on the ideal $I_{n,d}^{\uu,t}$ and the corresponding $\text{in}( I_{n,d}^{\uu,t})$. By linearity, we can split each minor $\e_{ijk}(\mathcal{M}_{n,d}^{\uu,t})$ in the following way
\begin{equation*}
\begin{split}
\begin{array}{|ccc|}
\vert\uu\vert-\hat{u}_i t & \vert\uu\vert-\hat{u}_j t & \vert\uu\vert-\hat{u}_k t \\
x_i & x_j & x_k \\
x_i^d & x_j^d & x_k^d \\
\end{array} &{}=
\vert\uu\vert\begin{array}{|ccc|}
1 & 1 & 1 \\
x_i & x_j & x_k \\
x_i^d & x_j^d & x_k^d \\
\end{array} - t\, \begin{array}{|ccc|}
\hat{u}_i & \hat{u}_j & \hat{u}_k \\
x_i & x_j & x_k \\
x_i^d & x_j^d & x_k^d \\
\end{array} ={}\\
&{} = \vert\uu\vert\mathrm{e}_{ijk}(\mathcal{M}_{n,d}^{\boldsymbol{1}}) - t\mathrm{e}_{ijk}(\mathcal{M}_{n,d}^{\hat{\uu}})
\end{split}
\end{equation*}
and we can also deduce that $\text{in}\big(\vert\uu\vert\mathrm{e}_{ijk}(\mathcal{M}_{n,d}^{\boldsymbol{1}}) - t\mathrm{e}_{ijk}(\mathcal{M}_{n,d}^{\hat{\uu}})\big) = \vert\uu\vert\mathrm{e}_{ijk}(\mathcal{M}_{n,d}^{\boldsymbol{1}}) $. Hence, 
\begin{equation*}
I_{n,d}^{\boldsymbol{1}} = \big(F_{n,d},\mathrm{e}_{ijk}(\mathcal{M}_{n,d}^{\boldsymbol{1}}),\ \forall\ 0\leqslant i<j<k\leqslant n) \subseteq \text{in}\big(I_{n,d}^{\uu,t}\big).
\end{equation*}
To prove that in fact equality holds, we can study the relations among these polynomials. The ideal $I_{n,d}^{\boldsymbol{1}}$ is the ideal of initial forms of $I_{n,d}^{\uu,t}$ if, and only if, each non-trivial syzygy between a pair of generators of $I_{n,d}^{\boldsymbol{1}}$ can be lifted to a syzygy of $I_{n,d}^{\uu,t}$. The polynomial defining the Fermat hypersurface is irreducible and it has no non-trivial syzygies with the other generators of $I^{\boldsymbol{1}}_{n,d}$. The other generators have four irreducible factors
\begin{equation*}
\mathrm{e}_{ijk}(\mathcal{M}^{\boldsymbol{1}}_{n,d}) = (x_i-x_j)(x_j-x_k)(x_k-x_i)\quad\sum_{\mathclap{\cramped{\begin{subarray}{c} (e_i,e_j,e_k)\in\ZZ^3_{\geq 0}\\e_i + e_j + e_k = d-2 \end{subarray}}}} \, x_i^{e_i} x_j^{e_j} x_k^{e_k}
\end{equation*}
so that two distinct generators $\mathrm{e}_{ijk}(\mathcal{M}^{\boldsymbol{1}}_{n,d})$ and $\mathrm{e}_{i'j'k'}(\mathcal{M}^{\boldsymbol{1}}_{n,d})$ have a non-trivial syzygy if, and only if, two of the three indices are equal. Let us consider $\mathrm{e}_{ijk}(\mathcal{M}^{\boldsymbol{1}}_{n,d})$ and $\mathrm{e}_{ijh}(\mathcal{M}^{\boldsymbol{1}}_{n,d})$. Without loss of generality, we may assume $i<j<k<h$.  The corresponding syzygy is
\begin{equation*}
\frac{\mathrm{e}_{ijh}(\mathcal{M}^{\boldsymbol{1}}_{n,d})}{x_i-x_j} \mathrm{e}_{ijk}(\mathcal{M}^{\boldsymbol{1}}_{n,d}) - \frac{\mathrm{e}_{ijk}(\mathcal{M}^{\boldsymbol{1}}_{n,d})}{x_i-x_j} \mathrm{e}_{ijh}(\mathcal{M}^{\boldsymbol{1}}_{n,d}) = 0.
\end{equation*}
 By direct computation, it is possible to check that the syzygy between these two generators of $I_{n,d}^{\boldsymbol{1}}$ can be lifted to the syzygy
\begin{equation*}
\begin{split}
&\frac{\mathrm{e}_{ijh}(\mathcal{M}^{\boldsymbol{1}}_{n,d})}{x_i-x_j} \mathrm{e}_{ijk}(\mathcal{M}^{\boldsymbol{u},t}_{n,d}) - \frac{\mathrm{e}_{ijk}(\mathcal{M}^{\boldsymbol{1}}_{n,d})}{x_i-x_j} \mathrm{e}_{ijh}(\mathcal{M}^{\boldsymbol{u},t}_{n,d}) = {}\\
&\hspace{0.5cm} {} = \vert\boldsymbol{u}\vert x_i x_j \Bigg(\quad\ \sum_{\mathclap{\cramped{\begin{subarray}{c} (e_i,e_j)\in\ZZ^2_{\geq 0}\\ e_i + e_j = d-2 \end{subarray}}}} \ x_i^{e_i} x_j^{e_j}\Bigg)\left( -\e_{ijk}(\mathcal{M}_{n,d}^{\uu,t}) + \e_{ijh}(\mathcal{M}_{n,d}^{\uu,t}) - \e_{ikh}(\mathcal{M}_{n,d}^{\uu,t}) + \e_{jkh}(\mathcal{M}_{n,d}^{\uu,t})\right)\\
\end{split}
\end{equation*}
among\hfill the\hfill generators\hfill of\hfill $I_{n,d}^{\uu,t}$.\hfill
Finally,\hfill $\K[x_0,\ldots,x_n,t]/I_{n,d}^{\uu,t}$\hfill is\hfill flat,\hfill as\hfill the\hfill graded\hfill module\\ $\gr\big(\K[x_0,\ldots,x_n,t]/I_{n,d}^{\uu,t}\big) = \K[x_0,\ldots,x_n,t]/I_{n,d}^{\boldsymbol{1}}$ is free over $\K[t]$ (the quotient does not depend on $t$).
\end{proof}

\begin{theorem}\label{prop:flatness} 
\begin{equation}\label{eq:mainEq}
\MLdeg F_{n,d} = \deg \left( I_{n,d}^{\boldsymbol{1}} \setminus \mathcal{H}\right) = \deg I_{n,d}^{\boldsymbol{1}} - \deg \left( I_{n,d}^{\boldsymbol{1}} + (x_0+\cdots+x_n) \right).
\end{equation}
Moreover, for every point $\vv \in \PP^n_{\yy}$ such that $\vert\vv\vert \neq 0$,
\[
\MLdeg F_{n,d} = \deg I_{n,d}^{\vv} - \deg \left( I_{n,d}^{\vv} + (x_0+\cdots+x_n) \right).
\]
\end{theorem}
\begin{proof}
Let $\uu \in \PP^n_{\yy}$ be a generic point for the projection $\pi_{\yy}:\mathcal{L}_{n,d} \rightarrow \PP^n_{\yy}$, i.e.~such that the length of the fiber $\pi_{\yy}^{-1}(\uu)$ is equal to the ML degree of $F_{n,d}$. We may assume by Proposition \ref{prop:daniele_smooth} that $\Proj \K[x]/I^{\uu}_{n,d}$ is smooth along the hyperplane $x_0 + \cdots + x_n = 0$,  $\vert \uu \vert \neq 0$ and $u_i \neq 0,\ \forall\ i = 0,\ldots,n$. Proving \eqref{eq:mainEq} is equivalent to prove $\deg I_{n,d}^{\uu} = \deg I_{n,d}^{\boldsymbol{1}}$. Indeed, by Lemma \ref{prop:davide} and Proposition \ref{prop:daniele_smooth} (Remark \ref{rk:also1}), we have $\deg ( I_{n,d}^{\uu}\setminus \mathcal{H}) = \deg I_{n,d}^{\uu} - \deg \big( I_{n,d}^{\uu} + (x_0+\cdots+x_n) \big)$, $\deg ( I_{n,d}^{\boldsymbol{1}}\setminus \mathcal{H}) = \deg I_{n,d}^{\boldsymbol{1}} - \deg \big( I_{n,d}^{\boldsymbol{1}} + (x_0+\cdots+x_n) \big)$
and $I_{n,d}^{\uu} + (x_0+\cdots+x_n) =I_{n,d}^{\boldsymbol{1}} + (x_0+\cdots+x_n)$ by Lemma \ref{prop:notDependU}.
The equality $\deg I_{n,d}^{\uu} = \deg I_{n,d}^{\boldsymbol{1}}$ follows from Lemma \ref{lem:flatness}. In fact, the ideals $I_{n,d}^{\boldsymbol{1}}$ and $I_{n,d}^{\uu}$ correspond respectively to the fibers over $t=0$ and $t=1$ of the family $\AA^1 \times_{\PP^n_{\boldsymbol{y}}} \mathcal{X} \rightarrow \AA^1$ and in the case of a family of zero-dimensional schemes with an irreducible reduced base, flatness is equivalent to the fact that all the fibers over closed points have the same degree \cite[III, Theorem 9.9]{Hartshorne}. 

To prove the last part of the statement, it suffices to repeat the same reasoning starting with $\MLdeg F_{n,d} = \deg I_{n,d}^{\boldsymbol{1}} - \deg \big( I_{n,d}^{\boldsymbol{1}} + (x_0+\cdots+x_n) \big)$.
\end{proof}

\begin{example}\label{ex:referee}
Consider the Fermat surface $F_{3,3} \subset \PP^3$. Its ML degree is $30$ and we can determine it by computing the multidegree of the likelihood correspondence $\mathcal{L}_{3,3} \subset \PP^3_{\xx} \times \PP^3_{\yy}$. Now, we look at the behavior at the points $\boldsymbol{1} = [1:1:1:1]$, $\uu = [2:3:5:6]$, $\vv_1 = [2:3:-7:7]$, $\vv_2 = [2:3:5:0]$, $\vv_3 = [1:3:-6:2]$ and  $\vv_4 = [1:-1:1:-1]$. By Theorem \ref{prop:flatness}, we know that
\[
\begin{split}
30 = \MLdeg F_{3,3} &{}= \deg \big(I_{3,3}^{\boldsymbol{1}} \setminus \mathcal{H} \big) = \deg I_{3,3}^{\boldsymbol{1}} - \deg \big( I_{3,3}^{\boldsymbol{1}} + (x_0+x_1+x_2+x_3) \big)= {}\\
&{}= \deg I_{3,3}^{\uu} - \deg \big( I_{3,3}^{\uu} + (x_0+x_1+x_2+x_3) \big) = {}\\
&{}= \deg I_{3,3}^{\vv_1} - \deg \big( I_{3,3}^{\vv_1} + (x_0+x_1+x_2+x_3) \big) = {}\\
&{}= \deg I_{3,3}^{\vv_2} - \deg \big( I_{3,3}^{\vv_2} + (x_0+x_1+x_2+x_3) \big),\\
\end{split}
\]
as $\vert \uu \vert \neq 0$, $\vert \vv_1 \vert \neq 0$ and $\vert \vv_2 \vert \neq 0$.
However, the theorem says nothing about the critical points of the likelihood functions $\ell_{\uu}$, $\ell_{\vv_1}$, $\ell_{\vv_2}$ not lying over $\mathcal{H}$. By direct computation, we can check that $\uu$ is \lq\lq generic\rq\rq\ for the projection map $\pi_{\yy}: \mathcal{L}_{3,3} \rightarrow \PP^3_{\yy}$, i.e.~$\deg I_{3,3}^{\uu} - \deg \big( I_{3,3}^{\uu} + (x_0+x_1+x_2+x_3) \big) = \deg \big(I_{3,3}^{\uu} \setminus \mathcal{H}\big)$, whereas $\ell_{\vv_1}$ and $\ell_{\vv_2}$ have more critical points lying on $\mathcal{H}$. More precisely, $\deg \big(I_{3,3}^{\vv_1}\setminus \mathcal{H}\big) = 28$ and $\deg \big(I_{3,3}^{\vv_2}\setminus \mathcal{H}\big) = 21$. To understand this behavior, we notice that the data vectors $\vv_1$ and $\vv_2$ do not satisfy the assumption of Theorem \ref{th:1.15HS} and that $\vv_2$ does not even satisfy the hypothesis of Lemma \ref{prop:davide}.

As $\vert \vv_3 \vert = \vert \vv_4 \vert = 0$, the last two points do not even satisfy the weaker assumption of Theorem \ref{prop:flatness} and can not be use to determine the ML degree of $F_{3,3}$. In the case of $\vv_3$, we check that $\ell_{\vv_3}$ has a finite number of critical points and that $\deg I_{3,3}^{\vv_3} = \deg I_{3,3}^{\boldsymbol{1}}$ (even if we can not deduce it from Lemma \ref{lem:flatness}), but Lemma \ref{prop:notDependU} does not apply. Indeed, $\deg \big(I_{3,3}^{\vv_3} + (x_0+x_1+x_2+x_3)) = 12 \neq \deg \big(I_{3,3}^{\boldsymbol{1}} + (x_0+x_1+x_2+x_3))$. 
The case of $\vv_4$ is even more special as the ideal $I_{3,3}^{\vv_4}$ describing the entire set of critical points of $\ell_{\vv_4}$ defines a pair of lines plus 22 points. These 22 points represent the set of critical points of $\ell_{\vv_4}$ not lying on $\mathcal{H}$.
\end{example}

\section{Determine solutions}\label{solutionfinding}

The result of the previous section allows us to restrict our attention to the ideal $I^{\boldsymbol{1}}_{n,d}$ generated by $f_{n,d}$ and by the $3\times3$ minors of the matrix
\begin{equation*}
\mathcal{M}^{\boldsymbol{1}}_{n,d}=
\left(\begin{array}{cccc}
1 & 1 & \ldots & 1 \\
x_0 & x_1 & \ldots & x_n \\
x_{0}^{d} & x_{1}^{d} & \ldots & x_{n}^{d}
\end{array}
\right).
\end{equation*}
It is clear that if two coordinates coincide, then any $3 \times 3$-minor of $\mathcal{M}_{n,d}^{\boldsymbol{1}}$ involving the two corresponding columns vanishes. Furthermore, we know that there are at least two distinct coordinates, since the point $[1:\ldots:1]$ does not lie on the Fermat hypersurface $F_{n,d}$. The next natural question is about finding the maximal number of distinct coordinates for a critical point. Consider a critical point with at least three distinct coordinates $p_i,p_j,p_k$. Since the $3\times 3$ minors of $\mathcal{M}_{n,d}^{\boldsymbol{1}}$ split into factors
\begin{equation}\label{eq:minor1ijk}
\det \left(\begin{array}{ccc}
1 & 1 & 1 \\
x_i & x_j & x_k \\
x_{i}^{d} & x_{j}^{d} & x_{k}^{d}
\end{array}
\right) = (x_i - x_j)(x_j - x_k)(x_k-x_i)\quad \sum_{\mathclap{\cramped{\begin{subarray}{c} (e_i,e_j,e_k)\in\ZZ^3_{\geq 0}\\e_i + e_j + e_k = d-2 \end{subarray}}}}\, x_i^{e_i} x_j^{e_j} x_k^{e_k},
\end{equation}
any $3$-tuple of distinct coordinates of a critical point is a solution to the polynomial 
\begin{equation}\label{eq:detMinor}
\sum_{\mathclap{\cramped{\begin{subarray}{c} (e_i,e_j,e_k)\in\ZZ^3_{\geq 0}\\e_i + e_j + e_k = d-2 \end{subarray}}}} \, x_i^{e_i} x_j^{e_j} x_k^{e_k}.
\end{equation}

\begin{lemma}\label{lem:diffValues}
Consider the Fermat hypersurface of degree $d$. The coordinates of each solution of $I_{n,d}^{\boldsymbol{1}}\setminus \mathcal{H}$ can take at most $\min\{d,n+1\}$ distinct values. 
\end{lemma}
\begin{proof}
Consider $k$ distinct values $p_1,\ldots,p_k$ and assume they appear as coordinates of a critical point. Clearly $k\leqslant n+1$. Furthermore, we see from polynomial \eqref{eq:detMinor} that once we fix $p_1$ and $p_2$, the other values $p_3,\ldots,p_k$ have to be roots of the univariate polynomial 
\begin{equation*}
\sum_{e_k=0}^{d-2} \left(\qquad\sum_{\mathclap{\cramped{e_i + e_j = d-2-e_k}}}\ \,p_1^{e_i} p_2^{e_j} \right) x_k^{e_k},
\end{equation*}
which has degree $d-2$. Thus, $k$ can be at most $d$.
\end{proof}

\begin{definition}
Let $\mathsf{a} := (a_1,\ldots,a_s)$ be an integer partition of $n+1$ such that $a_1 \geqslant \cdots \geqslant a_s \geqslant 1$ and  $2\leqslant s \leqslant \min\{d,n+1\}$. We say that a critical point $\boldsymbol{p}$ of the ML degree problem of the Fermat hypersurface $F_{n,d}$ is of type $(a_1,\ldots,a_s)$, or a \emph{$(a_1,\ldots,a_s)$-critical point}, if $\boldsymbol{p}$ has $s$ distinct coordinates and each distinct value $p_i$ appears $a_i$ times as coordinate of the point, for $1\leqslant i\leqslant s$.
\end{definition}

Let us recall some standard notation for partitions. If some integer $b$ is repeated $\alpha$ times in $\mathsf{a}$, we will write $b^{\alpha}$ instead of $b,\ldots,b$. Moreover, for any partition $\mathsf{a}$, we denote by $\alpha_{\mathsf{a}}$ the set of \lq\lq exponents\rq\rq. For instance, if $\mathsf{a} = (4,3,3,3,2,1,1) = (4,3^3,2,1^2)$, then $\alpha_{\mathsf{a}} = (1,3,1,2)$. The length of a partition $\mathsf{a}$ is the sum of all exponents $\vert\alpha_{\mathsf{a}}\vert$. In the previous example the partition has length 7.

\medskip

Let us define $\mathcal{P}_{n+1,d}$ as the set of partitions of $n+1$ of length at most $\min\{d,n+1\}$. We will present each element in $\mathcal{P}_{n+1,d}$ as a vector whose components form a sequence of weakly decreasing positive integers. 

Using the notation we introduced above, the property stated in Lemma \ref{lem:diffValues} translates in the following formula to compute the ML degree of the Fermat hypersurface:
\begin{equation*}
\MLdeg F_{n,d} = \quad\sum_{\mathclap{\cramped{\mathsf{a} \in \mathcal{P}_{n+1,d}}}}\, \#\{ \mathsf{a}\text{-critical points} \}.
\end{equation*}

Let $\mathsf{a} = (a_1,\ldots,a_s)$ be a partition in $\mathcal{P}_{n+1,d}$ and let $\alpha_{\mathsf{a}} = (\alpha_1,\ldots,\alpha_\sigma)$ be the corresponding set of exponents. We consider the integers
\begin{equation*}
 c_{\mathsf{a}} := \binom{n+1}{a_1} \cdot \binom{n+1-a_1}{a_2} \cdot\ldots\cdot \binom{n+1-(a_1+\ldots + a_{s-2})}{a_{s-1}} \text{ and } \alpha_{\mathsf{a}}! :=  \alpha_1! \cdot\ldots\cdot \alpha_\sigma!.
\end{equation*}
In addition, we want to determine the ideal that defines the subset of $\mathsf{a}$-critical points of $F_{n,d}$ by specifying the identifications among the coordinates. We can compute this ideal from $I_{n,d}^{\boldsymbol{1}}$ by adding the linear form $x_i - x_j$ if the $i$-th and $j$-th coordinates have to be equal and by saturating with respect to the linear form $x_k - x_h$ if the $k$-th and $h$-th coordinates have to be different. 
The identifications allow to rewrite the ideal in terms of $s$ variables, say $z_1,\ldots,z_s$, and the ideal written in terms of these variables (up to relabeling of the variables) does not depend on the identifications. Thus, we define
\begin{equation}\label{eq:partitionIdeal}
\begin{split}
& I_{\mathsf{a}}^d := \left(a_1z_1^d + \cdots + a_s z_s^d, \e_{ijk}(\mathcal{M}^{\boldsymbol{1}}_{s-1,d}),\ \forall\ 1 \leqslant i < j < k \leqslant s\right) : \Big(\prod_{i<j} (z_i - z_j)\Big),\quad \\
& I_{\mathsf{a}}^d \setminus \mathcal{H} := I_{\mathsf{a}}^d : (a_1z_1 + \cdots + a_s z_s),\qquad I_{\mathsf{a}}^d \cap \mathcal{H} := I_{\mathsf{a}}^d + (a_1z_1 + \cdots + a_s z_s),
\end{split}
\end{equation}
where $\e_{ijk}(\mathcal{M}^{\boldsymbol{1}}_{s-1,d})$ is the polynomial \eqref{eq:minor1ijk} in the variables $z_i$, $z_j$ and $z_k$.

\begin{theorem}\label{th:partitioning}
\begin{equation*}
\MLdeg F_{n,d}=\quad\sum_{\mathclap{\cramped{\mathsf{a}\in \mathcal{P}_{n+1,d}}}}\, c_{\mathsf{a}}\frac{\deg \big( I_{\mathsf{a}}^d \setminus \mathcal{H}\big)}{\alpha_{\mathsf{a}}!}=\quad\sum_{\mathclap{\cramped{\mathsf{a}\in \mathcal{P}_{n+1,d}}}}\, c_{\mathsf{a}}\frac{\deg I_{\mathsf{a}}^d - \deg \big( I_{\mathsf{a}}^d \cap \mathcal{H}\big)}{\alpha_{\mathsf{a}}!}
\end{equation*}
\end{theorem}
\begin{proof}
In order to prove the statement, we show that the number of $\mathsf{a}$-critical points equals $c_{\mathsf{a}}\frac{\deg (I_{\mathsf{a}}^d \setminus \mathcal{H})}{\alpha_{\mathsf{a}}!}$ for each $\mathsf{a} \in \mathcal{P}_{n+1,d}$.

Let $\boldsymbol{p} = [p_1:\ldots:p_s]$ be a solution of the ideal $I_{\mathsf{a}}^d \setminus \mathcal{H}$. It corresponds to the $\mathsf{a}$-critical points $\boldsymbol{q} = [q_0:\ldots:q_n]$ having $a_i$ coordinates equal to $p_i$ (up to a fixed scalar) for all $i$. As the symmetric group $\mathcal{S}_{n+1}$ acts on the solutions of $I_{n,d}^{\boldsymbol{1}} \setminus \mathcal{H}$, the number of points $\boldsymbol{q}$ corresponding to $\boldsymbol{p}$ is equal to the number of ways of partitioning the set of indices $\{0,\ldots,n\}$ in $s$ subsets with cardinality given by the entries of $\mathsf{a}$. This is exactly $c_{\mathsf{a}}$.

To complete the proof we need to determine how many \lq\lq different\rq\rq\ solutions the ideal $I_{\mathsf{a}}^d \setminus \mathcal{H}$ has. For instance, if $a_1 = a_2$ and $\boldsymbol{p} = [p_1:p_2:\ldots:p_s]$ is a solution of the ideal $I_{\mathsf{a}}^d \setminus \mathcal{H}$, then also $\boldsymbol{p}' = [p_2:p_1:\ldots:p_s]$ is a solution. But $\boldsymbol{p}$ and $\boldsymbol{p}'$ correspond to the same subset of $\mathsf{a}$-critical points with $n+1$ coordinates since we are considering the action of $\mathcal{S}_{n+1}$ on them.  Looking at the definition \eqref{eq:partitionIdeal} of $I_{\mathsf{a}}^d \setminus \mathcal{H}$, we notice that the ideals do not change if we swap two variables $z_i,z_j$ such that $a_i=a_j$. The repetitions of the entries of $\mathsf{a}$ are counted by the sequence $\alpha_{\mathsf{a}}$ and $I_{\mathsf{a}}^d \setminus \mathcal{H}$ is invariant under the action of the group $\mathcal{S}_{\alpha_1}\times\cdots\times\mathcal{S}_{\alpha_\sigma}$. Hence, the solutions we are interested in are the solutions of $I_{\mathsf{a}}^d \setminus \mathcal{H}$ modulo the action of this group of symmetries, i.e.~$\deg \big(I_{\mathsf{a}}^d \setminus \mathcal{H}\big)/\alpha_{\mathsf{a}}!$. The equality $\deg \big(I_{\mathsf{a}}^d \setminus \mathcal{H}\big) = \deg I_{\mathsf{a}}^d - \deg \big(I_{\mathsf{a}}^d \cap \mathcal{H}\big)$ follows directly from the equality $\deg \big(I_{n,d}^{\boldsymbol{1}} \setminus \mathcal{H}\big) = \deg I_{n,d}^{\boldsymbol{1}} - \deg \big(I_{n,d}^{\boldsymbol{1}} + (x_0+\cdots+c_n)\big)$.
\end{proof}

\begin{table}
\caption{The $\mathsf{a}$-critical points of the quartic Fermat hypersurface $F_{n,4}$ for $n=2,3,4$.}
\begin{center}
\begin{tikzpicture}[scale = 0.94]
\draw [-,thick] (0,0.5) -- (13,0.5);
\draw [-,thick] (0,-0.1) -- (13,-0.1);
\draw [-,thin] (0,-0.6) -- (13,-0.6);
\draw [-,thick] (0,-1.1) -- (13,-1.1);

\draw [-,thick] (10,-1.2) -- (13,-1.2) -- (13,-1.8) -- (10,-1.8) -- cycle;

\draw [-,thick] (0,0.5) -- (0,-1.1);
\draw [-,thin] (2,0.5) -- (2,-1.1);
\draw [-,thin] (4,0.5) -- (4,-1.1);
\draw [-,thin] (6,0.5) -- (6,-1.1);
\draw [-,thin] (8,0.5) -- (8,-1.1);
\draw [-,thin] (10,0.5) -- (10,-1.1);
\draw [-,thick] (13,0.5) -- (13,-1.1);

\node at (1,0.2) [] {\footnotesize$\mathsf{a} \in \mathcal{P}_{3,4}$};
\node at (3,0.2) [] {\footnotesize$\deg I_{\mathsf{a}}^4$};
\node at (5,0.2) [] {\footnotesize$\deg \big(I_{\mathsf{a}}^4\cap\mathcal{H}\big)$};
\node at (7,0.2) [] {\footnotesize$c_{\mathsf{a}}$};
\node at (9,0.2) [] {\footnotesize$\alpha_{\mathsf{a}}!$};
\node at (11.5,0.2) [] {\footnotesize$\mathsf{a}$-critical points};

\node at (1,-0.35) [] {\footnotesize$(2,1)$};
\node at (3,-0.35) [] {\footnotesize4};
\node at (5,-0.35) [] {\footnotesize0};
\node at (7,-0.35) [] {\footnotesize3};
\node at (9,-0.35) [] {\footnotesize1};
\node at (11.5,-0.35) [] {\footnotesize12};

\node at (1,-0.85) [] {\footnotesize$(1,1,1)$};
\node at (3,-0.85) [] {\footnotesize8};
\node at (5,-0.85) [] {\footnotesize2};
\node at (7,-0.85) [] {\footnotesize6};
\node at (9,-0.85) [] {\footnotesize6};
\node at (11.5,-0.85) [] {\footnotesize6};

\node at (9,-1.55) [] {\footnotesize$\MLdeg F_{2,4}$};
\node at (11.5,-1.5) [] {\small$18$};


\draw [-,thick] (0,-2.5) -- (13,-2.5);
\draw [-,thick] (0,-3.1) -- (13,-3.1);
\draw [-,thin] (0,-3.6) -- (13,-3.6);
\draw [-,thin] (0,-4.1) -- (13,-4.1);
\draw [-,thin] (0,-4.6) -- (13,-4.6);
\draw [-,thick] (0,-5.1) -- (13,-5.1);

\draw [-,thick] (0,-2.5) -- (0,-5.1);
\draw [-,thin] (2,-2.5) -- (2,-5.1);
\draw [-,thin] (4,-2.5) -- (4,-5.1);
\draw [-,thin] (6,-2.5) -- (6,-5.1);
\draw [-,thin] (8,-2.5) -- (8,-5.1);
\draw [-,thin] (10,-2.5) -- (10,-5.1);
\draw [-,thick] (13,-2.5) -- (13,-5.1);

\draw [-,thick] (10,-5.2) -- (13,-5.2) -- (13,-5.8) -- (10,-5.8) -- cycle;

\node at (1,-2.8) [] {\footnotesize$\mathsf{a} \in \mathcal{P}_{4,4}$};
\node at (3,-2.8) [] {\footnotesize$\deg I_{\mathsf{a}}^4$};
\node at (5,-2.8) [] {\footnotesize$\deg \big(I_{\mathsf{a}}^4\cap\mathcal{H}\big)$};
\node at (7,-2.8) [] {\footnotesize$c_{\mathsf{a}}$};
\node at (9,-2.8) [] {\footnotesize$\alpha_{\mathsf{a}}!$};
\node at (11.5,-2.8) [] {\footnotesize$\mathsf{a}$-critical points};

\node at (1,-3.35) [] {\footnotesize$(3,1)$};
\node at (3,-3.35) [] {\footnotesize4};
\node at (5,-3.35) [] {\footnotesize0};
\node at (7,-3.35) [] {\footnotesize4};
\node at (9,-3.35) [] {\footnotesize1};
\node at (11.5,-3.35) [] {\footnotesize16};

\node at (1,-3.85) [] {\footnotesize$(2,2)$};
\node at (3,-3.85) [] {\footnotesize4};
\node at (5,-3.85) [] {\footnotesize0};
\node at (7,-3.85) [] {\footnotesize6};
\node at (9,-3.85) [] {\footnotesize2};
\node at (11.5,-3.85) [] {\footnotesize12};

\node at (1,-4.35) [] {\footnotesize$(2,1,1)$};
\node at (3,-4.35) [] {\footnotesize8};
\node at (5,-4.35) [] {\footnotesize0};
\node at (7,-4.35) [] {\footnotesize12};
\node at (9,-4.35) [] {\footnotesize2};
\node at (11.5,-4.35) [] {\footnotesize48};

\node at (1,-4.85) [] {\footnotesize$(1,1,1,1)$};
\node at (3,-4.85) [] {\footnotesize8};
\node at (5,-4.85) [] {\footnotesize8};
\node at (7,-4.85) [] {\footnotesize24};
\node at (9,-4.85) [] {\footnotesize24};
\node at (11.5,-4.85) [] {\footnotesize0};

\node at (9,-5.55) [] {\footnotesize$\MLdeg F_{3,4}$};
\node at (11.5,-5.5) [] {\small$76$};


\draw [-,thick] (0,-6.5) -- (13,-6.5);
\draw [-,thick] (0,-7.1) -- (13,-7.1);
\draw [-,thin] (0,-7.6) -- (13,-7.6);
\draw [-,thin] (0,-8.1) -- (13,-8.1);
\draw [-,thin] (0,-8.6) -- (13,-8.6);
\draw [-,thin] (0,-9.1) -- (13,-9.1);
\draw [-,thick] (0,-9.6) -- (13,-9.6);

\draw [-,thick] (0,-6.5) -- (0,-9.6);
\draw [-,thin] (2,-6.5) -- (2,-9.6);
\draw [-,thin] (4,-6.5) -- (4,-9.6);
\draw [-,thin] (6,-6.5) -- (6,-9.6);
\draw [-,thin] (8,-6.5) -- (8,-9.6);
\draw [-,thin] (10,-6.5) -- (10,-9.6);
\draw [-,thick] (13,-6.5) -- (13,-9.6);

\draw [-,thick] (10,-9.7) -- (13,-9.7) -- (13,-10.3) -- (10,-10.3) -- cycle;

\node at (1,-6.8) [] {\footnotesize$\mathsf{a} \in \mathcal{P}_{5,4}$};
\node at (3,-6.8) [] {\footnotesize$\deg I_{\mathsf{a}}^4$};
\node at (5,-6.8) [] {\footnotesize$\deg \big(I_{\mathsf{a}}^4\cap\mathcal{H}\big)$};
\node at (7,-6.8) [] {\footnotesize$c_{\mathsf{a}}$};
\node at (9,-6.8) [] {\footnotesize$\alpha_{\mathsf{a}}!$};
\node at (11.5,-6.8) [] {\footnotesize$\mathsf{a}$-critical points};

\node at (1,-7.35) [] {\footnotesize$(4,1)$};
\node at (3,-7.35) [] {\footnotesize4};
\node at (5,-7.35) [] {\footnotesize0};
\node at (7,-7.35) [] {\footnotesize5};
\node at (9,-7.35) [] {\footnotesize1};
\node at (11.5,-7.35) [] {\footnotesize20};

\node at (1,-7.85) [] {\footnotesize$(3,2)$};
\node at (3,-7.85) [] {\footnotesize4};
\node at (5,-7.85) [] {\footnotesize0};
\node at (7,-7.85) [] {\footnotesize10};
\node at (9,-7.85) [] {\footnotesize1};
\node at (11.5,-7.85) [] {\footnotesize40};

\node at (1,-8.35) [] {\footnotesize$(3,1,1)$};
\node at (3,-8.35) [] {\footnotesize8};
\node at (5,-8.35) [] {\footnotesize0};
\node at (7,-8.35) [] {\footnotesize20};
\node at (9,-8.35) [] {\footnotesize2};
\node at (11.5,-8.35) [] {\footnotesize80};

\node at (1,-8.85) [] {\footnotesize$(2,2,1)$};
\node at (3,-8.85) [] {\footnotesize8};
\node at (5,-8.85) [] {\footnotesize0};
\node at (7,-8.85) [] {\footnotesize30};
\node at (9,-8.85) [] {\footnotesize2};
\node at (11.5,-8.85) [] {\footnotesize120};

\node at (1,-9.35) [] {\footnotesize$(2,1,1,1)$};
\node at (3,-9.35) [] {\footnotesize8};
\node at (5,-9.35) [] {\footnotesize2};
\node at (7,-9.35) [] {\footnotesize60};
\node at (9,-9.35) [] {\footnotesize6};
\node at (11.5,-9.35) [] {\footnotesize60};

\node at (9,-10.05) [] {\footnotesize$\MLdeg F_{4,4}$};
\node at (11.5,-10) [] {\small$320$};

\end{tikzpicture}
\end{center}
\end{table}

We conclude this section with a more detailed description of the ideal $I_{\mathsf{a}}^d$, where we try to make the saturation by the linear forms $z_i-z_j$ more explicit. Let us introduce some notation that will be useful later. For any set of indices $\mathcal{I} = \{i_1,\ldots,i_r\}$ of variables $z_j$ and for any non-negative integer $m$, let
\begin{equation*}
\boldsymbol{z}_{\mathcal{I}}^{(0)} := 1\qquad\text{and}\qquad \boldsymbol{z}_{\mathcal{I}}^{(m)} :=\quad \sum_{\mathclap{\cramped{\begin{subarray}{c} (e_1,\ldots,e_r) \in \ZZ^{r}_{\geqslant 0}\\
e_1 + \cdots + e_r = m \end{subarray}}}}\, z_{i_1}^{e_1}\cdots z_{i_r}^{e_r},\ \text{for } m > 0.
\end{equation*} 

\begin{lemma}\label{lem:deletingLinear}
Consider a set of $s$ indices $\mathcal{I}$ and fix two distinct elements $h$ and $k$ in $\mathcal{I}$. For every positive integer $m$,
\begin{equation*}
\boldsymbol{z}_{\mathcal{I}\setminus k}^{(m)} - \boldsymbol{z}_{\mathcal{I}\setminus h}^{(m)} = (z_h - z_k)\boldsymbol{z}_{\mathcal{I}}^{(m-1)}. 
\end{equation*}
\end{lemma}
\begin{proof}
By looking at $\boldsymbol{z}_{\mathcal{I}\setminus k}^{(m)}$ as an univariate polynomial in the variable $z_h$ and $\boldsymbol{z}_{\mathcal{I}\setminus h}^{(m)}$ as an univariate polynomial in the variable $z_k$, we can write
\begin{equation*}
 \boldsymbol{z}_{\mathcal{I}\setminus k}^{(m)} = \sum_{e = 0}^{m} \left(\boldsymbol{z}_{\mathcal{I}\setminus h,k}^{(m-e)} \right) z_h^e \qquad \text{and}\qquad \boldsymbol{z}_{\mathcal{I}\setminus h}^{(m)} = \sum_{e = 0}^{m} \left(\boldsymbol{z}_{\mathcal{I}\setminus h,k}^{(m-e)} \right) z_k^e.
\end{equation*} 
Hence,
\begin{equation*}
\begin{split}
\boldsymbol{z}_{\mathcal{I}\setminus k}^{(m)} - \boldsymbol{z}_{\mathcal{I}\setminus h}^{(m)} &{} = \sum_{e = 0}^{m} \left[\left(\boldsymbol{z}_{\mathcal{I}\setminus h,k}^{(m-e)} \right) (z_h^e - z_k^e)\right] = {}\\
&{} = (z_h - z_k)\sum_{e = 1}^{m} \left[\left(\boldsymbol{z}_{\mathcal{I}\setminus h,k}^{(m-e)} \right) \boldsymbol{z}_{h,k}^{(e-1)}\right] = (z_h-z_k)\boldsymbol{z}_{\mathcal{I}}^{(m-1)}.
\end{split}
\end{equation*}
\end{proof}

\begin{proposition}\label{prop:shapeI}
Let $\mathsf{a} = (a_1,\ldots,a_s)$ be a partition in $\mathcal{P}_{n+1,d}$. 
\begin{enumerate}[(i)]
\item\label{it:shapeI_i} If $s = 2$, then $I_{\mathsf{a}}^d = (a_1 z_1^d + a_2 z_2^d)$.
\item\label{it:shapeI_ii}  If $s > 2$, then $I_{\mathsf{a}}^d$ contains the polynomials $\boldsymbol{z}^{(d-w+1)}_{\mathcal{I}}$ for all multi-indices $\mathcal{I} \subset \{1,\ldots,s\}$ with $w$ elements, for all $w=3,\ldots,s$.
\end{enumerate}
\end{proposition}
\begin{proof}
\emph{(\ref{it:shapeI_i})} If the coordinates of a critical point have only two distinct values, then the $3\times 3$ minors of $\mathcal{M}_{n,d}^{\boldsymbol{1}}$ automatically vanish because all $3\times 3$ submatrices of $\mathcal{M}_{n,d}^{\boldsymbol{1}}$ have at most 2 linearly independent columns. Moreover, the unique solution to $a_1 z_1^d + a_2z_2^d$ with $z_1=z_2$ is given by the pair $(0,0)$ that does not define a point. Therefore, we can avoid the saturation with respect to $z_1-z_2$.

\emph{(\ref{it:shapeI_ii})} Let us proceed iteratively: we start with the generators of $I_{\mathsf{a}}^d$
\begin{equation*}
\e_{ijk}\left(\mathcal{M}_{s-1,d}^{\boldsymbol{1}}\right) = (z_i-z_j)(z_j-z_k)(z_k-z_i) \boldsymbol{z}_{i,j,k}^{(d-2)}.
\end{equation*}
Since we compute these polynomials for every $3$-tuple $i,j,k$ and we saturate with respect to all linear forms given by the difference of two variables, the statement is true for $w=3$. For $3 < w \leqslant s$, it suffices to apply Lemma \ref{lem:deletingLinear}.
\end{proof}

\section{Closed formulas for the ML degree of special Fermat hypersurfaces}\label{closedformula}

In this section, we give two closed formulas for computing the ML degree of the Fermat hypersurfaces $F_{n,2}$ and $F_{2,d}$. In the case of $F_{n,2}$, the set $\mathcal{P}_{n+1,2}$ contains only partitions $\mathsf{a}$ of length 2 and the ideal $I_{\mathsf{a}}^2$ is very simple (Proposition \ref{prop:shapeI}\emph{(\ref{it:shapeI_i})}), while in the case $F_{2,d}$ we can successfully apply a topological criterion which works because we have only $3$ variables.

\subsection{\texorpdfstring{The ML degree of $F_{n,2}$}{The ML degree of F\_{n,2}}}
 
 We start counting the number of solutions of the ideal $I_{\mathsf{a}}^d\setminus\mathcal{H}$ for partitions $\mathsf{a} = (a_1,a_2)$ of length $2$.
 \begin{lemma}\label{lem:solAB}
Let $\mathsf{a}=(a_1,a_2)$ be a partition of $\mathcal{P}_{n+1,d}$. 
\begin{equation*}
\deg \big(I_{\mathsf{a}}^d\setminus\mathcal{H}\big) = \begin{cases}
d-1,& \text{if } a_1 = a_2 \text{ and } d \text{ odd},\\
d,& \text{otherwise}.
\end{cases}
\end{equation*}
\end{lemma}
\begin{proof}
We start considering the solutions of the equation $a_1z_1^d + a_2z_2^d = 0$. Without loss of generality, we can assume $z_2 = 1$, so that we have $d$ solutions
\begin{equation*}
\left[ \sqrt[d]{\frac{a_2}{a_1}}\, \mathrm{e}^{\mathrm{i}\frac{2k+1}{d}\pi}: 1\right],\qquad	k = 0,\ldots,d-1.
\end{equation*}
We need to discard every solution lying on the line $a_1 z_1 + a_2z_2 = 0$. Such a point exists only in the case $a_1 = a_2$ and $\mathrm{e}^{\mathrm{i}\frac{2k+1}{d}\pi} = -1$, i.e.~$d = 2k+1$.
\end{proof}

We now apply the strategy used in the proof of Theorem \ref{th:partitioning} to count the $(a_1,a_2)$-critical points of $I_{n,d}^{\boldsymbol{1}} \setminus \mathcal{H}$ for all $(a_1,a_2) \in \mathcal{P}_{n+1,d}$.
\begin{proposition}\label{prop:totAB}
The total number of $\mathsf{a}$-critical points of the ideal $I^{\boldsymbol{1}}_{n,d}\setminus \mathcal{H}$ as $\mathsf{a}$ varies over all pairs $(a_1,a_2) \in \mathcal{P}_{n+1,d}$ is
\begin{equation}
d(2^n-1) - \frac{1}{2}\binom{n+1}{\frac{n+1}{2}}\text{ if $d$ and $n$ are odd},\qquad d(2^n-1) \text{ otherwise}. 
\end{equation}
\end{proposition}
\begin{proof}
If $n+1$ is odd, for each partition $(a_1,a_2) \in \mathcal{P}_{n+1,d}$ of length 2, we have $\alpha_{(a_1,a_2)}! = 1$ (as $a_1 > a_2$) and $\deg \big(I_{(a_1,a_2)}^d\setminus\mathcal{H}\big) = d$ (Lemma \ref{lem:solAB}), so that the total number of $(a_1,a_2)$-solutions of $I_{n,d}^{\boldsymbol{1}} \setminus \mathcal{H}$ is
\begin{equation*}
\begin{split}
\sum_{\mathclap{\cramped{(a_1,a_2) \in \mathcal{P}_{n+1,d}}}}\, c_{(a_1,a_2)}\frac{\deg \big(I^{d}_{(a_1,a_2)}\setminus\mathcal{H}\big)}{\alpha_{(a_1,a_2)}!} = \quad\sum_{\mathclap{\cramped{\begin{subarray}{c} a_1+a_2 = n+1\\a_1 > a_2 > 0\end{subarray}}}}\ \: d\binom{n+1}{a_1} =  d \sum_{i=1}^{n/2} \binom{n+1}{i} = {}&\\
{} = d\frac{1}{2} \sum_{i=1}^{n} \binom{n+1}{i} = d\frac{1}{2} \left(\sum_{i=0}^{n+1} \binom{n+1}{i} - 2\right) ={}& d(2^n-1).
\end{split}
\end{equation*}

If $n+1$ is even, there is also the partition $(\frac{n+1}{2},\frac{n+1}{2})$ for which $\alpha_{(\frac{n+1}{2},\frac{n+1}{2})}! = 2$. Assume $d$ even, so that $\deg \big(I_{(\frac{n+1}{2},\frac{n+1}{2})}^d\setminus\mathcal{H}\big) = d$. The total number of $(a_1,a_2)$-critical points is
\begin{equation*}
\begin{split}
\sum_{\mathclap{\cramped{(a_1,a_2) \in \mathcal{P}_{n+1,d}}}}\, c_{(a_1,a_2)}\frac{\deg \big(I^{d}_{(a_1,a_2)}\setminus\mathcal{H}\big)}{\alpha_{(a_1,a_2)}!} &{}= \quad\sum_{\mathclap{\cramped{\begin{subarray}{c} a_1+a_2 = n+1\\a_1 > a_2 > 0\end{subarray}}}}\ \: d\binom{n+1}{a_1} + \frac{d}{2}\binom{n+1}{\tfrac{n+1}{2}} ={}\\  
&{}  = d \sum_{i=1}^{\frac{n-1}{2}} \binom{n+1}{i} + \frac{d}{2}\binom{n+1}{\tfrac{n+1}{2}} =  \frac{d}{2} \sum_{i=1}^{n} \binom{n+1}{i}  = {}\\
&{} = \frac{d}{2} \left(\sum_{i=0}^{n} \binom{n+1}{i} - 2\right) = d(2^n-1).
\end{split}
\end{equation*}
Finally, if $d$ is odd, we have one less solution of $I_{\big(\frac{n+1}{2},\frac{n+1}{2}\big)}^d \setminus\mathcal{H}$ so that the total number of $(a_1,a_2)$-critical points is decreased by $c_{\big(\frac{n+1}{2},\frac{n+1}{2}\big)}/\alpha_{\big(\frac{n+1}{2},\frac{n+1}{2}\big)}! = \frac{1}{2}\binom{n+1}{\tfrac{n+1}{2}}$.
\end{proof}

\begin{corollary}\label{prop:MLdegD2}
The ML degree of the Fermat hypersurface $F_{n,2}$ is $2^{n+1} - 2$.
\end{corollary}
\begin{proof}
If $d=2$, all the partitions of $\mathcal{P}_{n+1,2}$ have length 2, so that the number of $(a_1,a_2)$-critical points determined in Proposition \ref{prop:totAB} is exactly the ML degree of $F_{n,d}$.
\end{proof}

Notice that this corollary proves that the Fermat hypersurface of degree $2$ has the same ML degree of the generic hypersurface of degree $2$ (cf.~\cite[Theorem 1.10, Example 1.11]{HuhSturmfels}). This is the expected result, as the quadric hypersurfaces are classified by the rank of the associated symmetric matrix and the quadric Fermat hypersurface has maximal rank as the generic quadric hypersurface. In general, this property does not hold. For instance, in the case discussed in Example \ref{ex:referee}, the Fermat surface $F_{3,3}$ has ML degree 30, while the ML degree of the generic cubic surface is $39$. Notice that in order to obtain a cubic surface with ML degree equal to 39 it suffices to randomly choose the coefficients of the monomials in the equation of the Fermat surface.

\subsection{\texorpdfstring{The ML degree of $F_{2,d}$}{The ML degree of F\_{2,d}}}

Applying Lemma \ref{lem:solAB} and Theorem \ref{th:partitioning} to the case $n=2$, we obtain
\begin{equation*}
\MLdeg F_{2,d} = c_{(2,1)} \frac{\deg \big(I_{(2,1)}^d\setminus\mathcal{H}\big)}{\alpha_{(2,1)}!} + c_{(1,1,1)} \frac{\deg \big(I_{(1,1,1)}^d\setminus\mathcal{H}\big)}{\alpha_{(1,1,1)}!} = 3d + \deg \big(I_{(1,1,1)}^d\setminus\mathcal{H}\big).
\end{equation*}
Furthermore, since in the case of partitions of length 3 there is a unique $3\times 3$ minor, the solutions of $I_{(1,1,1)}^d\setminus\mathcal{H}$ are contained in the intersection of $\boldsymbol{z}_{1,2,3}^{(d-2)} = 0$ with $z_1^d + z_2^d + z_3^d=0$. By Bezout's theorem, we know that there are $d(d-2)$ solutions (counted with multiplicity). This gives the upper bound of the ML degree $3d + d(d-2) = d^2 + d$.

In this case, instead of trying to determine explicitly the solutions of $I_{(1,1,1)}^d \setminus \mathcal{H}$, we resort to a special case of a theorem that correlates the ML degree of a variety with its signed topological Euler characteristic (see \cite{CHKS,Huh,HuhSturmfels}). To make the paper as self-contained as possible we report the theorem below. Recall that $\mathcal{H}$ denotes the distinguished arrangement of hyperplanes.

\begin{theorem}[{\cite[Theorem 1.1]{HuhSturmfels}}]\label{th:teoSturmfels}
Let $X$ be a smooth curve of degree $d$ in $\PP^2$, and $\a=\#(X\cap \HH)$ the number of its points on the distinguished arrangement. Then the ML degree of $X$ equals $d^2-3d+\a$.
\end{theorem}

As the Fermat curve is smooth, we need to calculate the number $\a_d=\#(F_{2,d}\cap \HH)$ for $d > 1$. 

\begin{proposition}
The ML degree of the Fermat curve $F_{2,d}$ is
\begin{equation*}
\MLdeg F_{2,d}=
\begin{cases}
d^2+d, &\text{if } d\equiv 0,2 \bmod 6, \\
d^2+d-3,&\text{if }  d\equiv 3,5 \bmod 6, \\
d^2+d-2,&\text{if }  d\equiv 4 \bmod 6, \\
d^2+d-5,&\text{if }  d\equiv 1 \bmod 6,
\end{cases}
\end{equation*}
\end{proposition}
\begin{proof}
Let us examine the intersection of $F_{n,d}$ with $x_0=0$, $x_1=0$, $x_2=0$ and $x_0+x_1+x_0=0$. Notice that the points $[1:0:0]$, $[0:1:0]$ and $[0:0:1]$ (lying on two different lines $x_i=x_j=0$) are not points of $F_{2,d}$. By Bezout's theorem, the intersection $F_{n,d} \cap \{x_0 = 0\}$ consists of $d$ points counted with multiplicity. In this case, all the points are simple as they correspond to the solutions of the equation $x_1^d+x_2^d=0$. By symmetry, we conclude that $\#(F_{2,d}\cap\{x_i=0\})=d$ for all $i=0,1,2$.

Now consider the line $\{x_0+x_1+x_2=0\}$ and assume $x_2 \neq 0$. We have
\small
\begin{equation*}
\begin{cases}
x_0^d+x_1^d+1=0\\
x_0+x_1+1=0
\end{cases}\ \quad\quad
\begin{cases}
x_0^d+(-1)^d(x_0+1)^d+1=0\\
x_1=-1-x_0
\end{cases}.
\end{equation*}
\normalsize
As multiple roots of the intersection $F_{2,d} \cap \{x_0 + x_1 + x_2 = 0\}$ corresponds to multiple roots of the polynomial $f_d(x_0):=x_0^d+(-1)^d(x_0+1)^d+1\in \CC[x_0]$, we try to determine the common factors of $f_d(x_0)$ and its derivative $f'_d(x_0) = d \big( x_0^{d-1} + (-1)^d (x_0+1)^{d-1}\big)$: 
\small
\begin{equation*}
\begin{cases}
x_0^d+(-1)^d(x_0+1)^d+1=0\\
x_0^{d-1} + (-1)^d (x_0+1)^{d-1} = 0
\end{cases}\
\begin{cases}
x_0^d+(x_0+1)(- x_0^{d-1})+1=0\\
(-1)^d (x_0+1)^{d-1} = - x_0^{d-1}
\end{cases}\
\begin{cases}
x_0^{d-1} = 1\\
(x_0+1)^{d-1} = (-1)^{d-1}
\end{cases}.
\end{equation*}
\normalsize
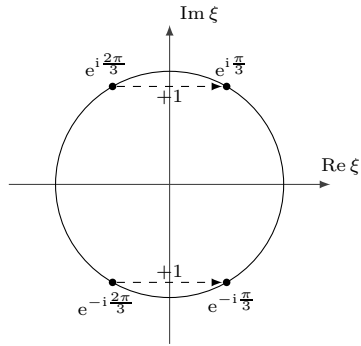
\begin{figure}[!ht]
\begin{center}
\begin{tikzpicture}[>=latex,scale=0.75]
\node (a) at (-3,0) [] {};
\node (b) at (3,0) [] {};
\node at (3,0) [above] {\tiny $\text{Re}\, \xi$};
\node (c) at (0,-3) [] {};
\node (d) at (0,3) [] {};
\node at (0,3) [right] {\tiny $\text{Im}\, \xi$};

\draw [->,thin,black!75] (a) -- (b);
\draw [->,thin,black!75] (c) -- (d);

\draw [] (0,0) circle (2);

\node (w1) at (-1,1.732) [fill=black,circle,inner sep=1pt] {};
\node (w11) at (1,1.732) [fill=black,circle,inner sep=1pt] {};

\node at (-1.1,1.732) [above] {\tiny $\mathrm{e}^{\mathrm{i}\frac{2\pi}{3}}$};
\node at (1.1,1.732) [above] {\tiny $\mathrm{e}^{\mathrm{i}\frac{\pi}{3}}$};

\node at (-1.1,-1.732) [below] {\tiny $\mathrm{e}^{-\mathrm{i}\frac{2\pi}{3}}$};
\node at (1.1,-1.732) [below] {\tiny $\mathrm{e}^{-\mathrm{i}\frac{\pi}{3}}$};

\node (w2) at (-1,-1.732) [fill=black,circle,inner sep=1pt] {};
\node (w21) at (1,-1.732) [fill=black,circle,inner sep=1pt] {};

\draw [dashed,->] (w1) -- node[below,inner sep=1pt]{\tiny $+1$} (w11);
\draw [dashed,->] (w2) -- node[above,inner sep=1pt]{\tiny $+1$} (w21);
\end{tikzpicture}
\caption{\label{fig:plus1} The complex numbers $\xi$ such that $\vert\xi\vert = \vert\xi+1\vert = 1$.}
\end{center}
\end{figure}
Now we observe that if $x_0\in \CC$ is a solution to the system, then $|x_0|=|x_0+1|=1$, where $\vert\xi\vert$ is the usual Euclidean norm of a complex number $\xi$. The only complex numbers satisfying such relations are the third root of unity $\omega_3=\e^{\mathrm{i}\frac{2\pi}{3}}$ and $\overline{\omega_{3}}$ (see Figure \ref{fig:plus1}). Thus, we can check directly whether $\omega_3$ and $\overline{\omega_3}$ are solutions. Substituting $x_0 = \omega_3$, we obtain
\begin{equation*}
\begin{cases}
\omega_3^{d-1}=1\\
(\omega_3+1)^{d-1}=(-1)^{d-1}
\end{cases}
\quad {\text{\scriptsize $(\omega_3+1=-\overline{\omega_{3}})$}}\qquad
\begin{cases}
\omega_3^{d-1}=1\\
\overline{\omega_{3}}^{d-1}=1
\end{cases}.
\end{equation*}
These equations are satisfied if, and only if, $3$ divides $d-1$, i.e.~$d\equiv 1\bmod 3$. In this case the multiple roots are exactly two, each with multiplicity two, since the second derivative $f_d''(x_0)$ does not vanish at $\omega_3$ and $\overline{\omega_3}$. Since the three coordinates of the multiple roots are the three solutions to the equation $\xi^3 = 1$, we notice that changing the affine chart does not lead to any other multiple roots. It follows that
\begin{equation*}
\#(F_{2,d}\cap \{x_0+x_1+x_2=0\})=
\begin{cases}
d, & \text{if } d\not\equiv 1 \bmod 3, \\
d-2, & \text{if } d\equiv 1 \bmod 3. 
\end{cases}
\end{equation*}
To conclude the proof, we need to count how many simple points of $F_{2,d}\cap \{x_0+x_1+x_2=0\}$ lie also on a line $x_i = 0$. The points lying on $F_{2,d}\cap \{x_0+x_1+x_2=0\} \cap \{x_0=0\}$ satisfy the system
\begin{equation*}
\begin{cases}
x_0 = 0\\
x_1 = -x_2\\
\left((-1)^d + 1\right)x_2^d = 0
\end{cases}.
\end{equation*}
If $d$ is even, there are no solutions,  while if $d$ is odd, we get the point $[0:1:-1]$. We conclude that
\begin{equation*}
\begin{split}
\# (F_{2,d} \cap \mathcal{H}) = {}& 3 \#\left( F_{2,d} \cap \{x_0=0\}\right) \\
&\qquad{} + \#\left( F_{2,d} \cap \{x_0+x_1+x_2=0\}\right) \\
&\qquad\qquad{} - 3\#\left( F_{2,d} \cap \{x_0+x_1+x_2=0\} \cap \{x_0 = 0\} \right) = {}\\
{} = {}& 
\left\{\begin{array}{ccccccl}
3d &+& d, && &&\text{if } d\not\equiv 1 \bmod 3 \text{ and }  d \text{ even}, \\                     
3d &+&d &-&3,  &&\text{if } d\not\equiv 1 \bmod 3 \text{ and }  d \text{ odd}, \\ 
3d &+& d-2,&&  &&\text{if } d\equiv 1 \bmod 3 \text{ and }  d \text{ even}, \\ 
3d &+& d-2&-&3,  &&\text{if } d\equiv 1 \bmod 3 \text{ and }  d \text{ odd}, \\ 
\end{array}\right.\\
{}={}& \left\{\begin{array}{lcl}
4d, && \text{if } d\equiv 0,2 \bmod 6,\\
4d-3, && \text{if } d\equiv 3,5 \bmod 6,\\
4d-2, &&\text{if }  d\equiv 4 \bmod 6,\\
4d-5, &&\text{if }  d\equiv 1 \bmod 6.
\end{array}\right.
\end{split}
\end{equation*}
The formula for the ML degree follows by Theorem \ref{th:teoSturmfels}.
\end{proof}

\section{Computational results}

This last section is intended to report the computational results we obtained. We provide a comparison among the possible procedures to get the ML degree of Fermat hypersurfaces that we explored and developed throughout the paper. We discussed three main methods: determining the multidegree of the likelihood correspondence, determining the degree of the projection $\pi_{\yy}:\mathcal{L}_{n,d} \rightarrow \PP^n_{\yy}$ considering the fiber of a generic point of $\PP^n_{\yy}$ and partitioning the solutions of $I_{n,d}^{\boldsymbol{1}}\setminus\mathcal{H}$ according to the number of distinct coordinates (Theorem \ref{th:partitioning}). To all these methods, Lemma \ref{prop:davide} applies, so that we can choose to compute the saturation of an ideal with respect to a linear form or to compute the degree of two ideals, the one not saturated and the one intersected with the linear form. Hence, we compare the following six strategies (see Table \ref{tab:2345}).
\begin{description}
\item[Strategy 1 (simple \lq\lq multidegree\rq\rq)] determine the ML degree of $F_{n,d}$ as the leading coefficient of the multidegree polynomial \eqref{eq:bidegree} of the likelihood correspondence $\mathcal{L}_{n,d}$.
\item[Strategy 2 (\lq\lq multidegree\rq\rq\ by difference)] apply Lemma \ref{prop:davide} and Proposition \ref{prop:daniele_smooth} and determine the multidegree polynomial of $\mathcal{L}_{n,d}$ as the difference between the multidegree polynomial of the variety defined by $I_{n,d}^{\yy}$ and the multidegree polynomial of the variety defined by $I_{n,d}^{\yy} + (x_0 + \cdots + x_n)$.
\item[Strategy 3 (simple \lq\lq random data\rq\rq)] determine the ML degree of $F_{n,d}$ as the degree of the projection $\pi_{\yy}:\mathcal{L}_{n,d} \rightarrow \PP^n_{\yy}$, i.e.~computing the degree of the ideal $I_{n,d}^{\uu}\setminus \mathcal{H}$ for a randomly chosen $\uu \in \PP^n_{\yy}$.
\item[Strategy 4 (\lq\lq random data\rq\rq\ by difference)] apply Lemma \ref{prop:davide} and Proposition \ref{prop:daniele_smooth} and compute $\deg \big(I_{n,d}^{\uu}\setminus \mathcal{H}\big)$ as $\deg I_{n,d}^{\uu} - \deg \big(I_{n,d}^{\uu} + (x_0+\ldots+x_n)\big)$ for a randomly chosen $\uu \in \PP^n_{\yy}$.
\item[Strategy 5 (simple \lq\lq partitioning\rq\rq)] apply Theorem \ref{prop:flatness} and Theorem \ref{th:partitioning} and determine $\MLdeg F_{n,d}$ from the degree of the ideals $I_{\mathsf{a}}^d  \setminus\mathcal{H}$ defining the $\mathsf{a}$-critical points.
\item[Strategy 6 (\lq\lq partitioning\rq\rq\ by difference)] apply Lemma \ref{prop:davide}, Proposition \ref{prop:daniele_smooth}, Theorem \ref{prop:flatness} and Theorem \ref{th:partitioning} and compute $\deg\big(I_{\mathsf{a}}^d \setminus \mathcal{H}\big)$ as $\deg I_{\mathsf{a}}^d - \deg\big(I_{\mathsf{a}}^d \cap \mathcal{H}\big)$.
\end{description}

\begin{table}[H]
\begin{center}
\caption{\small\label{tab:2345} An experimental comparison of the different strategies. All the algorithms for the different strategies  are implemented with the software \emph{Macaulay2} \cite{M2}. The code is developed in the package \texttt{MLdegreeFermatHypersurface.m2} available at \href{http://www.paololella.it/EN/Publications.html}{\texttt{http://www.paololella.it/EN/Publications.html}}. A complete report of the tests can be found at the same address. The algorithms have been run on a Intel(R) Xeon(R) 2.60 GHz processor. We considered the Fermat hypersurfaces in the ranges $2 \leqslant n \leqslant 9,\ 2 \leqslant d \leqslant 9$ with a limit for completion of $10^4$ seconds of cpu-time. For more efficient computations, polynomial rings with coefficients in a finite field have been exploited; the correctness of the results has been checked by repeating the same process using fields with different characteristic.}
\end{center}

\begin{center}
\subfloat[][\textbf{Simple \lq\lq multidegree\rq\rq.}]
{\label{tab:cf1}
\begin{tikzpicture}[scale=0.4]

\node at (-.5,0) [] {};
\node at (.5,0) [] {};
\filldraw[red!13!yellow] (0,0) rectangle (1,-1);
\filldraw[red!13!yellow] (1,0) rectangle (2,-1);
\filldraw[red!13!yellow] (2,0) rectangle (3,-1);
\filldraw[red!13!yellow] (3,0) rectangle (4,-1);
\filldraw[red!13!yellow] (4,0) rectangle (5,-1);
\filldraw[red!14!yellow] (5,0) rectangle (6,-1);
\filldraw[red!13!yellow] (6,0) rectangle (7,-1);
\filldraw[red!14!yellow] (7,0) rectangle (8,-1);


\filldraw[red!13!yellow] (0,-1) rectangle (1,-2);
\filldraw[red!17!yellow] (1,-1) rectangle (2,-2);
\filldraw[red!19!yellow] (2,-1) rectangle (3,-2);
\filldraw[red!33!yellow] (3,-1) rectangle (4,-2);
\filldraw[red!24!yellow] (4,-1) rectangle (5,-2);
\filldraw[red!50!yellow] (5,-1) rectangle (6,-2);
\filldraw[red!38!yellow] (6,-1) rectangle (7,-2);
\filldraw[red!62!yellow] (7,-1) rectangle (8,-2);


\filldraw[red!14!yellow] (0,-2) rectangle (1,-3);
\filldraw[red!57!yellow] (1,-2) rectangle (2,-3);
\filldraw[red!76!yellow] (2,-2) rectangle (3,-3);
\filldraw[red!91!yellow] (3,-2) rectangle (4,-3);
\filldraw[red!88!yellow] (4,-2) rectangle (5,-3);


\filldraw[red!24!yellow] (0,-3) rectangle (1,-4);


\filldraw[red!50!yellow] (0,-4) rectangle (1,-5);


\filldraw[red!74!yellow] (0,-5) rectangle (1,-6);


\node at (-1,1) [] {\footnotesize $F_{n,d}$};
\node at (-1.5,-4) [] {\footnotesize $n$};
\node at (4,1.5) [] {\footnotesize $d$};

\node at (-0.5,-0.5) [] {\footnotesize $2$};
\node at (-0.5,-1.5) [] {\footnotesize $3$};
\node at (-0.5,-2.5) [] {\footnotesize $4$};
\node at (-0.5,-3.5) [] {\footnotesize $5$};
\node at (-0.5,-4.5) [] {\footnotesize $6$};
\node at (-0.5,-5.5) [] {\footnotesize $7$};
\node at (-0.5,-6.5) [] {\footnotesize $8$};
\node at (-0.5,-7.5) [] {\footnotesize $9$};

\node at (0.5,0.5) [] {\footnotesize $2$};
\node at (1.5,0.5) [] {\footnotesize $3$};
\node at (2.5,0.5) [] {\footnotesize $4$};
\node at (3.5,0.5) [] {\footnotesize $5$};
\node at (4.5,0.5) [] {\footnotesize $6$};
\node at (5.5,0.5) [] {\footnotesize $7$};
\node at (6.5,0.5) [] {\footnotesize $8$};
\node at (7.5,0.5) [] {\footnotesize $9$};

\draw [-] (-1,-1) -- (8,-1);
\draw [-] (-1,-2) -- (8,-2);
\draw [-] (-1,-3) -- (5,-3);
\draw [-] (-1,-4) -- (1,-4);
\draw [-] (-1,-5) -- (1,-5);
\draw [-] (-1,-6) -- (1,-6);
\draw [-] (-1,-7) -- (0,-7);
\draw [-] (0,1) -- (8,1);

\draw [-] (1,1) -- (1,-6);
\draw [-] (2,1) -- (2,-3);
\draw [-] (3,1) -- (3,-3);
\draw [-] (4,1) -- (4,-3);
\draw [-] (5,1) -- (5,-3);
\draw [-] (6,1) -- (6,-2);
\draw [-] (7,1) -- (7,-2);
\draw [-] (-1,0) -- (-1,-8);

\draw [-,thick] (0,2) -- (0,-8);
\draw [-,thick] (-2,0) -- (8,0);

\draw[-,thick] (-2,2) -- (8,2) -- (8,-8) -- (-2,-8) -- cycle;
\end{tikzpicture}
}
\qquad\qquad
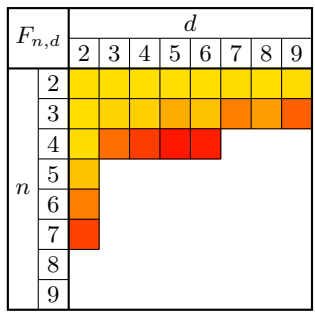
\subfloat[][\textbf{\lq\lq Multi-degree\rq\rq\ by difference.}]
{\label{tab:cf1plus} 
\begin{tikzpicture}[scale=0.4]

\node at (-3.25,0) [] {};
\node at (9.25,0) [] {};

\filldraw[red!13!yellow] (0,0) rectangle (1,-1);
\filldraw[red!13!yellow] (1,0) rectangle (2,-1);
\filldraw[red!13!yellow] (2,0) rectangle (3,-1);
\filldraw[red!13!yellow] (3,0) rectangle (4,-1);
\filldraw[red!13!yellow] (4,0) rectangle (5,-1);
\filldraw[red!13!yellow] (5,0) rectangle (6,-1);
\filldraw[red!13!yellow] (6,0) rectangle (7,-1);
\filldraw[red!13!yellow] (7,0) rectangle (8,-1);


\filldraw[red!14!yellow] (0,-1) rectangle (1,-2);
\filldraw[red!13!yellow] (1,-1) rectangle (2,-2);
\filldraw[red!13!yellow] (2,-1) rectangle (3,-2);
\filldraw[red!13!yellow] (3,-1) rectangle (4,-2);
\filldraw[red!14!yellow] (4,-1) rectangle (5,-2);
\filldraw[red!14!yellow] (5,-1) rectangle (6,-2);
\filldraw[red!15!yellow] (6,-1) rectangle (7,-2);
\filldraw[red!15!yellow] (7,-1) rectangle (8,-2);


\filldraw[red!13!yellow] (0,-2) rectangle (1,-3);
\filldraw[red!16!yellow] (1,-2) rectangle (2,-3);
\filldraw[red!18!yellow] (2,-2) rectangle (3,-3);
\filldraw[red!43!yellow] (3,-2) rectangle (4,-3);
\filldraw[red!36!yellow] (4,-2) rectangle (5,-3);
\filldraw[red!58!yellow] (5,-2) rectangle (6,-3);
\filldraw[red!52!yellow] (6,-2) rectangle (7,-3);
\filldraw[red!85!yellow] (7,-2) rectangle (8,-3);


\filldraw[red!15!yellow] (0,-3) rectangle (1,-4);
\filldraw[red!56!yellow] (1,-3) rectangle (2,-4);
\filldraw[red!48!yellow] (2,-3) rectangle (3,-4);
\filldraw[red!82!yellow] (4,-3) rectangle (5,-4);


\filldraw[red!20!yellow] (0,-4) rectangle (1,-5);
\filldraw[red!88!yellow] (2,-4) rectangle (3,-5);


\filldraw[red!32!yellow] (0,-5) rectangle (1,-6);


\filldraw[red!49!yellow] (0,-6) rectangle (1,-7);


\filldraw[red!67!yellow] (0,-7) rectangle (1,-8);


\node at (-1,1) [] {\footnotesize $F_{n,d}$};
\node at (-1.5,-4) [] {\footnotesize $n$};
\node at (4,1.5) [] {\footnotesize $d$};

\node at (-0.5,-0.5) [] {\footnotesize $2$};
\node at (-0.5,-1.5) [] {\footnotesize $3$};
\node at (-0.5,-2.5) [] {\footnotesize $4$};
\node at (-0.5,-3.5) [] {\footnotesize $5$};
\node at (-0.5,-4.5) [] {\footnotesize $6$};
\node at (-0.5,-5.5) [] {\footnotesize $7$};
\node at (-0.5,-6.5) [] {\footnotesize $8$};
\node at (-0.5,-7.5) [] {\footnotesize $9$};

\node at (0.5,0.5) [] {\footnotesize $2$};
\node at (1.5,0.5) [] {\footnotesize $3$};
\node at (2.5,0.5) [] {\footnotesize $4$};
\node at (3.5,0.5) [] {\footnotesize $5$};
\node at (4.5,0.5) [] {\footnotesize $6$};
\node at (5.5,0.5) [] {\footnotesize $7$};
\node at (6.5,0.5) [] {\footnotesize $8$};
\node at (7.5,0.5) [] {\footnotesize $9$};

\draw [-] (-1,-1) -- (8,-1);
\draw [-] (-1,-2) -- (8,-2);
\draw [-] (-1,-3) -- (8,-3);
\draw [-] (-1,-4) -- (3,-4);
\draw [-] (4,-4) -- (5,-4);
\draw [-] (-1,-5) -- (1,-5);
\draw [-] (2,-5) -- (3,-5);
\draw [-] (-1,-6) -- (1,-6);
\draw [-] (-1,-7) -- (1,-7);
\draw [-] (0,1) -- (8,1);

\draw [-] (1,1) -- (1,-8);
\draw [-] (2,1) -- (2,-5);
\draw [-] (3,1) -- (3,-5);
\draw [-] (4,1) -- (4,-4);
\draw [-] (5,1) -- (5,-4);
\draw [-] (6,1) -- (6,-3);
\draw [-] (7,1) -- (7,-3);
\draw [-] (-1,0) -- (-1,-8);

\draw [-,thick] (0,2) -- (0,-8);
\draw [-,thick] (-2,0) -- (8,0);

\draw[-,thick] (-2,2) -- (8,2) -- (8,-8) -- (-2,-8) -- cycle;
\end{tikzpicture}
}
\end{center}

\begin{center}
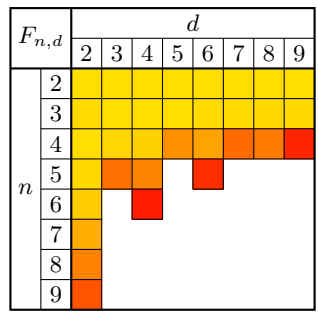
\subfloat[][\textbf{Simple \lq\lq random data\rq\rq.}]
{\label{tab:cf2}
\begin{tikzpicture}[scale=0.4]

\node at (-.5,0) [] {};
\node at (.5,0) [] {};

\filldraw[red!11!yellow] (0,0) rectangle (1,-1);
\filldraw[red!10!yellow] (1,0) rectangle (2,-1);
\filldraw[red!11!yellow] (2,0) rectangle (3,-1);
\filldraw[red!11!yellow] (3,0) rectangle (4,-1);
\filldraw[red!13!yellow] (4,0) rectangle (5,-1);
\filldraw[red!12!yellow] (5,0) rectangle (6,-1);
\filldraw[red!18!yellow] (6,0) rectangle (7,-1);
\filldraw[red!19!yellow] (7,0) rectangle (8,-1);


\filldraw[red!11!yellow] (0,-1) rectangle (1,-2);
\filldraw[red!13!yellow] (1,-1) rectangle (2,-2);
\filldraw[red!25!yellow] (2,-1) rectangle (3,-2);
\filldraw[red!36!yellow] (3,-1) rectangle (4,-2);
\filldraw[red!48!yellow] (4,-1) rectangle (5,-2);
\filldraw[red!57!yellow] (5,-1) rectangle (6,-2);
\filldraw[red!66!yellow] (6,-1) rectangle (7,-2);
\filldraw[red!73!yellow] (7,-1) rectangle (8,-2);


\filldraw[red!14!yellow] (0,-2) rectangle (1,-3);
\filldraw[red!33!yellow] (1,-2) rectangle (2,-3);
\filldraw[red!58!yellow] (2,-2) rectangle (3,-3);
\filldraw[red!79!yellow] (3,-2) rectangle (4,-3);
\filldraw[red!98!yellow] (4,-2) rectangle (5,-3);


\filldraw[red!29!yellow] (0,-3) rectangle (1,-4);
\filldraw[red!59!yellow] (1,-3) rectangle (2,-4);
\filldraw[red!95!yellow] (2,-3) rectangle (3,-4);


\filldraw[red!44!yellow] (0,-4) rectangle (1,-5);
\filldraw[red!93!yellow] (1,-4) rectangle (2,-5);


\filldraw[red!64!yellow] (0,-5) rectangle (1,-6);


\filldraw[red!86!yellow] (0,-6) rectangle (1,-7);


\filldraw[red!100!yellow] (0,-7) rectangle (1,-8);


\node at (-1,1) [] {\footnotesize $F_{n,d}$};
\node at (-1.5,-4) [] {\footnotesize $n$};
\node at (4,1.5) [] {\footnotesize $d$};

\node at (-0.5,-0.5) [] {\footnotesize $2$};
\node at (-0.5,-1.5) [] {\footnotesize $3$};
\node at (-0.5,-2.5) [] {\footnotesize $4$};
\node at (-0.5,-3.5) [] {\footnotesize $5$};
\node at (-0.5,-4.5) [] {\footnotesize $6$};
\node at (-0.5,-5.5) [] {\footnotesize $7$};
\node at (-0.5,-6.5) [] {\footnotesize $8$};
\node at (-0.5,-7.5) [] {\footnotesize $9$};

\node at (0.5,0.5) [] {\footnotesize $2$};
\node at (1.5,0.5) [] {\footnotesize $3$};
\node at (2.5,0.5) [] {\footnotesize $4$};
\node at (3.5,0.5) [] {\footnotesize $5$};
\node at (4.5,0.5) [] {\footnotesize $6$};
\node at (5.5,0.5) [] {\footnotesize $7$};
\node at (6.5,0.5) [] {\footnotesize $8$};
\node at (7.5,0.5) [] {\footnotesize $9$};

\draw [-] (-1,-1) -- (8,-1);
\draw [-] (-1,-2) -- (8,-2);
\draw [-] (-1,-3) -- (5,-3);
\draw [-] (-1,-4) -- (3,-4);
\draw [-] (-1,-5) -- (2,-5);
\draw [-] (-1,-6) -- (1,-6);
\draw [-] (-1,-7) -- (1,-7);
\draw [-] (0,1) -- (8,1);

\draw [-] (1,1) -- (1,-8);
\draw [-] (2,1) -- (2,-5);
\draw [-] (3,1) -- (3,-4);
\draw [-] (4,1) -- (4,-3);
\draw [-] (5,1) -- (5,-3);
\draw [-] (6,1) -- (6,-2);
\draw [-] (7,1) -- (7,-2);
\draw [-] (-1,0) -- (-1,-8);

\draw [-,thick] (0,2) -- (0,-8);
\draw [-,thick] (-2,0) -- (8,0);

\draw[-,thick] (-2,2) -- (8,2) -- (8,-8) -- (-2,-8) -- cycle;
\end{tikzpicture}
}
\qquad\qquad
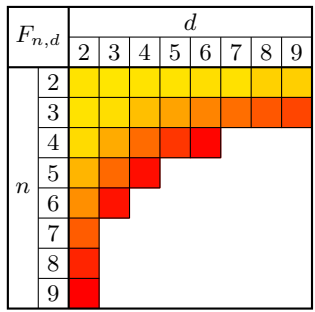
\subfloat[][\textbf{\lq\lq Random data\rq\rq\ by difference.}]
{\label{tab:cf2plus} 
\begin{tikzpicture}[scale=0.4]

\node at (-3.25,0) [] {};
\node at (9.25,0) [] {};

\filldraw[red!11!yellow] (0,0) rectangle (1,-1);
\filldraw[red!11!yellow] (1,0) rectangle (2,-1);
\filldraw[red!11!yellow] (2,0) rectangle (3,-1);
\filldraw[red!10!yellow] (3,0) rectangle (4,-1);
\filldraw[red!11!yellow] (4,0) rectangle (5,-1);
\filldraw[red!11!yellow] (5,0) rectangle (6,-1);
\filldraw[red!11!yellow] (6,0) rectangle (7,-1);
\filldraw[red!11!yellow] (7,0) rectangle (8,-1);


\filldraw[red!11!yellow] (0,-1) rectangle (1,-2);
\filldraw[red!10!yellow] (1,-1) rectangle (2,-2);
\filldraw[red!11!yellow] (2,-1) rectangle (3,-2);
\filldraw[red!11!yellow] (3,-1) rectangle (4,-2);
\filldraw[red!11!yellow] (4,-1) rectangle (5,-2);
\filldraw[red!12!yellow] (5,-1) rectangle (6,-2);
\filldraw[red!13!yellow] (6,-1) rectangle (7,-2);
\filldraw[red!14!yellow] (7,-1) rectangle (8,-2);


\filldraw[red!11!yellow] (0,-2) rectangle (1,-3);
\filldraw[red!12!yellow] (1,-2) rectangle (2,-3);
\filldraw[red!17!yellow] (2,-2) rectangle (3,-3);
\filldraw[red!26!yellow] (3,-2) rectangle (4,-3);
\filldraw[red!35!yellow] (4,-2) rectangle (5,-3);
\filldraw[red!46!yellow] (5,-2) rectangle (6,-3);
\filldraw[red!53!yellow] (6,-2) rectangle (7,-3);
\filldraw[red!65!yellow] (7,-2) rectangle (8,-3);


\filldraw[red!13!yellow] (0,-3) rectangle (1,-4);
\filldraw[red!25!yellow] (1,-3) rectangle (2,-4);
\filldraw[red!51!yellow] (2,-3) rectangle (3,-4);
\filldraw[red!69!yellow] (3,-3) rectangle (4,-4);
\filldraw[red!85!yellow] (4,-3) rectangle (5,-4);
\filldraw[red!96!yellow] (5,-3) rectangle (6,-4);


\filldraw[red!19!yellow] (0,-4) rectangle (1,-5);
\filldraw[red!54!yellow] (1,-4) rectangle (2,-5);
\filldraw[red!86!yellow] (2,-4) rectangle (3,-5);


\filldraw[red!33!yellow] (0,-5) rectangle (1,-6);
\filldraw[red!84!yellow] (1,-5) rectangle (2,-6);


\filldraw[red!52!yellow] (0,-6) rectangle (1,-7);


\filldraw[red!71!yellow] (0,-7) rectangle (1,-8);


\node at (-1,1) [] {\footnotesize $F_{n,d}$};
\node at (-1.5,-4) [] {\footnotesize $n$};
\node at (4,1.5) [] {\footnotesize $d$};

\node at (-0.5,-0.5) [] {\footnotesize $2$};
\node at (-0.5,-1.5) [] {\footnotesize $3$};
\node at (-0.5,-2.5) [] {\footnotesize $4$};
\node at (-0.5,-3.5) [] {\footnotesize $5$};
\node at (-0.5,-4.5) [] {\footnotesize $6$};
\node at (-0.5,-5.5) [] {\footnotesize $7$};
\node at (-0.5,-6.5) [] {\footnotesize $8$};
\node at (-0.5,-7.5) [] {\footnotesize $9$};

\node at (0.5,0.5) [] {\footnotesize $2$};
\node at (1.5,0.5) [] {\footnotesize $3$};
\node at (2.5,0.5) [] {\footnotesize $4$};
\node at (3.5,0.5) [] {\footnotesize $5$};
\node at (4.5,0.5) [] {\footnotesize $6$};
\node at (5.5,0.5) [] {\footnotesize $7$};
\node at (6.5,0.5) [] {\footnotesize $8$};
\node at (7.5,0.5) [] {\footnotesize $9$};

\draw [-] (-1,-1) -- (8,-1);
\draw [-] (-1,-2) -- (8,-2);
\draw [-] (-1,-3) -- (8,-3);
\draw [-] (-1,-4) -- (6,-4);
\draw [-] (-1,-5) -- (3,-5);
\draw [-] (-1,-6) -- (2,-6);
\draw [-] (-1,-7) -- (1,-7);
\draw [-] (0,1) -- (8,1);

\draw [-] (1,1) -- (1,-8);
\draw [-] (2,1) -- (2,-6);
\draw [-] (3,1) -- (3,-5);
\draw [-] (4,1) -- (4,-4);
\draw [-] (5,1) -- (5,-4);
\draw [-] (6,1) -- (6,-4);
\draw [-] (7,1) -- (7,-3);
\draw [-] (-1,0) -- (-1,-8);

\draw [-,thick] (0,2) -- (0,-8);
\draw [-,thick] (-2,0) -- (8,0);

\draw[-,thick] (-2,2) -- (8,2) -- (8,-8) -- (-2,-8) -- cycle;
\end{tikzpicture}
}
\end{center}

\begin{center}
\subfloat[][\textbf{Simple \lq\lq partitioning\rq\rq.}]
{\label{tab:cf3}
\begin{tikzpicture}[scale=0.4]

\node at (-0.5,0) [] {};
\node at (0.5,0) [] {};

\filldraw[red!11!yellow] (0,0) rectangle (1,-1);
\filldraw[red!13!yellow] (1,0) rectangle (2,-1);
\filldraw[red!14!yellow] (2,0) rectangle (3,-1);
\filldraw[red!14!yellow] (3,0) rectangle (4,-1);
\filldraw[red!14!yellow] (4,0) rectangle (5,-1);
\filldraw[red!15!yellow] (5,0) rectangle (6,-1);
\filldraw[red!16!yellow] (6,0) rectangle (7,-1);
\filldraw[red!19!yellow] (7,0) rectangle (8,-1);


\filldraw[red!12!yellow] (0,-1) rectangle (1,-2);
\filldraw[red!14!yellow] (1,-1) rectangle (2,-2);
\filldraw[red!17!yellow] (2,-1) rectangle (3,-2);
\filldraw[red!18!yellow] (3,-1) rectangle (4,-2);
\filldraw[red!22!yellow] (4,-1) rectangle (5,-2);
\filldraw[red!30!yellow] (5,-1) rectangle (6,-2);
\filldraw[red!37!yellow] (6,-1) rectangle (7,-2);
\filldraw[red!48!yellow] (7,-1) rectangle (8,-2);


\filldraw[red!12!yellow] (0,-2) rectangle (1,-3);
\filldraw[red!16!yellow] (1,-2) rectangle (2,-3);
\filldraw[red!18!yellow] (2,-2) rectangle (3,-3);
\filldraw[red!24!yellow] (3,-2) rectangle (4,-3);
\filldraw[red!36!yellow] (4,-2) rectangle (5,-3);
\filldraw[red!50!yellow] (5,-2) rectangle (6,-3);
\filldraw[red!70!yellow] (6,-2) rectangle (7,-3);
\filldraw[red!84!yellow] (7,-2) rectangle (8,-3);


\filldraw[red!13!yellow] (0,-3) rectangle (1,-4);
\filldraw[red!17!yellow] (1,-3) rectangle (2,-4);
\filldraw[red!21!yellow] (2,-3) rectangle (3,-4);
\filldraw[red!28!yellow] (3,-3) rectangle (4,-4);
\filldraw[red!41!yellow] (4,-3) rectangle (5,-4);
\filldraw[red!69!yellow] (5,-3) rectangle (6,-4);
\filldraw[red!94!yellow] (6,-3) rectangle (7,-4);


\filldraw[red!13!yellow] (0,-4) rectangle (1,-5);
\filldraw[red!18!yellow] (1,-4) rectangle (2,-5);
\filldraw[red!24!yellow] (2,-4) rectangle (3,-5);
\filldraw[red!31!yellow] (3,-4) rectangle (4,-5);
\filldraw[red!48!yellow] (4,-4) rectangle (5,-5);
\filldraw[red!75!yellow] (5,-4) rectangle (6,-5);


\filldraw[red!13!yellow] (0,-5) rectangle (1,-6);
\filldraw[red!20!yellow] (1,-5) rectangle (2,-6);
\filldraw[red!26!yellow] (2,-5) rectangle (3,-6);
\filldraw[red!34!yellow] (3,-5) rectangle (4,-6);
\filldraw[red!53!yellow] (4,-5) rectangle (5,-6);
\filldraw[red!86!yellow] (5,-5) rectangle (6,-6);


\filldraw[red!14!yellow] (0,-6) rectangle (1,-7);
\filldraw[red!21!yellow] (1,-6) rectangle (2,-7);
\filldraw[red!28!yellow] (2,-6) rectangle (3,-7);
\filldraw[red!37!yellow] (3,-6) rectangle (4,-7);
\filldraw[red!57!yellow] (4,-6) rectangle (5,-7);
\filldraw[red!90!yellow] (5,-6) rectangle (6,-7);


\filldraw[red!14!yellow] (0,-7) rectangle (1,-8);
\filldraw[red!23!yellow] (1,-7) rectangle (2,-8);
\filldraw[red!30!yellow] (2,-7) rectangle (3,-8);
\filldraw[red!40!yellow] (3,-7) rectangle (4,-8);
\filldraw[red!61!yellow] (4,-7) rectangle (5,-8);
\filldraw[red!94!yellow] (5,-7) rectangle (6,-8);


\node at (-1,1) [] {\footnotesize $F_{n,d}$};
\node at (-1.5,-4) [] {\footnotesize $n$};
\node at (4,1.5) [] {\footnotesize $d$};

\node at (-0.5,-0.5) [] {\footnotesize $2$};
\node at (-0.5,-1.5) [] {\footnotesize $3$};
\node at (-0.5,-2.5) [] {\footnotesize $4$};
\node at (-0.5,-3.5) [] {\footnotesize $5$};
\node at (-0.5,-4.5) [] {\footnotesize $6$};
\node at (-0.5,-5.5) [] {\footnotesize $7$};
\node at (-0.5,-6.5) [] {\footnotesize $8$};
\node at (-0.5,-7.5) [] {\footnotesize $9$};

\node at (0.5,0.5) [] {\footnotesize $2$};
\node at (1.5,0.5) [] {\footnotesize $3$};
\node at (2.5,0.5) [] {\footnotesize $4$};
\node at (3.5,0.5) [] {\footnotesize $5$};
\node at (4.5,0.5) [] {\footnotesize $6$};
\node at (5.5,0.5) [] {\footnotesize $7$};
\node at (6.5,0.5) [] {\footnotesize $8$};
\node at (7.5,0.5) [] {\footnotesize $9$};

\draw [-] (-1,-1) -- (8,-1);
\draw [-] (-1,-2) -- (8,-2);
\draw [-] (-1,-3) -- (8,-3);
\draw [-] (-1,-4) -- (7,-4);
\draw [-] (-1,-5) -- (6,-5);
\draw [-] (-1,-6) -- (6,-6);
\draw [-] (-1,-7) -- (6,-7);
\draw [-] (0,1) -- (8,1);

\draw [-] (1,1) -- (1,-8);
\draw [-] (2,1) -- (2,-8);
\draw [-] (3,1) -- (3,-8);
\draw [-] (4,1) -- (4,-8);
\draw [-] (5,1) -- (5,-8);
\draw [-] (6,1) -- (6,-8);
\draw [-] (7,1) -- (7,-4);
\draw [-] (-1,0) -- (-1,-8);

\draw [-,thick] (0,2) -- (0,-8);
\draw [-,thick] (-2,0) -- (8,0);

\draw[-,thick] (-2,2) -- (8,2) -- (8,-8) -- (-2,-8) -- cycle;
\end{tikzpicture}
}
\qquad\qquad
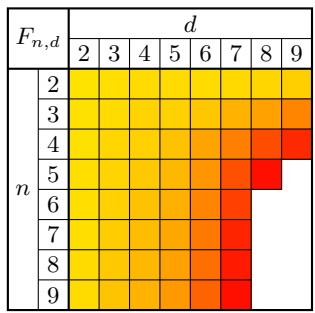
\subfloat[][\textbf{\lq\lq Partitioning\rq\rq\ by difference.}]
{\label{tab:cf3plus} 
\begin{tikzpicture}[scale=0.4]

\node at (-3.25,0) [] {};
\node at (9.25,0) [] {};

\filldraw[red!11!yellow] (0,0) rectangle (1,-1);
\filldraw[red!14!yellow] (1,0) rectangle (2,-1);
\filldraw[red!14!yellow] (2,0) rectangle (3,-1);
\filldraw[red!14!yellow] (3,0) rectangle (4,-1);
\filldraw[red!13!yellow] (4,0) rectangle (5,-1);
\filldraw[red!14!yellow] (5,0) rectangle (6,-1);
\filldraw[red!14!yellow] (6,0) rectangle (7,-1);
\filldraw[red!15!yellow] (7,0) rectangle (8,-1);


\filldraw[red!11!yellow] (0,-1) rectangle (1,-2);
\filldraw[red!13!yellow] (1,-1) rectangle (2,-2);
\filldraw[red!16!yellow] (2,-1) rectangle (3,-2);
\filldraw[red!18!yellow] (3,-1) rectangle (4,-2);
\filldraw[red!21!yellow] (4,-1) rectangle (5,-2);
\filldraw[red!27!yellow] (5,-1) rectangle (6,-2);
\filldraw[red!34!yellow] (6,-1) rectangle (7,-2);
\filldraw[red!42!yellow] (7,-1) rectangle (8,-2);


\filldraw[red!12!yellow] (0,-2) rectangle (1,-3);
\filldraw[red!15!yellow] (1,-2) rectangle (2,-3);
\filldraw[red!18!yellow] (2,-2) rectangle (3,-3);
\filldraw[red!24!yellow] (3,-2) rectangle (4,-3);
\filldraw[red!35!yellow] (4,-2) rectangle (5,-3);
\filldraw[red!50!yellow] (5,-2) rectangle (6,-3);
\filldraw[red!70!yellow] (6,-2) rectangle (7,-3);
\filldraw[red!84!yellow] (7,-2) rectangle (8,-3);


\filldraw[red!12!yellow] (0,-3) rectangle (1,-4);
\filldraw[red!17!yellow] (1,-3) rectangle (2,-4);
\filldraw[red!21!yellow] (2,-3) rectangle (3,-4);
\filldraw[red!27!yellow] (3,-3) rectangle (4,-4);
\filldraw[red!39!yellow] (4,-3) rectangle (5,-4);
\filldraw[red!69!yellow] (5,-3) rectangle (6,-4);
\filldraw[red!92!yellow] (6,-3) rectangle (7,-4);


\filldraw[red!13!yellow] (0,-4) rectangle (1,-5);
\filldraw[red!18!yellow] (1,-4) rectangle (2,-5);
\filldraw[red!23!yellow] (2,-4) rectangle (3,-5);
\filldraw[red!30!yellow] (3,-4) rectangle (4,-5);
\filldraw[red!46!yellow] (4,-4) rectangle (5,-5);
\filldraw[red!75!yellow] (5,-4) rectangle (6,-5);


\filldraw[red!13!yellow] (0,-5) rectangle (1,-6);
\filldraw[red!19!yellow] (1,-5) rectangle (2,-6);
\filldraw[red!25!yellow] (2,-5) rectangle (3,-6);
\filldraw[red!34!yellow] (3,-5) rectangle (4,-6);
\filldraw[red!49!yellow] (4,-5) rectangle (5,-6);
\filldraw[red!83!yellow] (5,-5) rectangle (6,-6);


\filldraw[red!13!yellow] (0,-6) rectangle (1,-7);
\filldraw[red!21!yellow] (1,-6) rectangle (2,-7);
\filldraw[red!28!yellow] (2,-6) rectangle (3,-7);
\filldraw[red!36!yellow] (3,-6) rectangle (4,-7);
\filldraw[red!52!yellow] (4,-6) rectangle (5,-7);
\filldraw[red!86!yellow] (5,-6) rectangle (6,-7);


\filldraw[red!12!yellow] (0,-7) rectangle (1,-8);
\filldraw[red!20!yellow] (1,-7) rectangle (2,-8);
\filldraw[red!29!yellow] (2,-7) rectangle (3,-8);
\filldraw[red!38!yellow] (3,-7) rectangle (4,-8);
\filldraw[red!56!yellow] (4,-7) rectangle (5,-8);
\filldraw[red!90!yellow] (5,-7) rectangle (6,-8);


\node at (-1,1) [] {\footnotesize $F_{n,d}$};
\node at (-1.5,-4) [] {\footnotesize $n$};
\node at (4,1.5) [] {\footnotesize $d$};

\node at (-0.5,-0.5) [] {\footnotesize $2$};
\node at (-0.5,-1.5) [] {\footnotesize $3$};
\node at (-0.5,-2.5) [] {\footnotesize $4$};
\node at (-0.5,-3.5) [] {\footnotesize $5$};
\node at (-0.5,-4.5) [] {\footnotesize $6$};
\node at (-0.5,-5.5) [] {\footnotesize $7$};
\node at (-0.5,-6.5) [] {\footnotesize $8$};
\node at (-0.5,-7.5) [] {\footnotesize $9$};

\node at (0.5,0.5) [] {\footnotesize $2$};
\node at (1.5,0.5) [] {\footnotesize $3$};
\node at (2.5,0.5) [] {\footnotesize $4$};
\node at (3.5,0.5) [] {\footnotesize $5$};
\node at (4.5,0.5) [] {\footnotesize $6$};
\node at (5.5,0.5) [] {\footnotesize $7$};
\node at (6.5,0.5) [] {\footnotesize $8$};
\node at (7.5,0.5) [] {\footnotesize $9$};

\draw [-] (-1,-1) -- (8,-1);
\draw [-] (-1,-2) -- (8,-2);
\draw [-] (-1,-3) -- (8,-3);
\draw [-] (-1,-4) -- (7,-4);
\draw [-] (-1,-5) -- (6,-5);
\draw [-] (-1,-6) -- (6,-6);
\draw [-] (-1,-7) -- (6,-7);
\draw [-] (0,1) -- (8,1);

\draw [-] (1,1) -- (1,-8);
\draw [-] (2,1) -- (2,-8);
\draw [-] (3,1) -- (3,-8);
\draw [-] (4,1) -- (4,-8);
\draw [-] (5,1) -- (5,-8);
\draw [-] (6,1) -- (6,-8);
\draw [-] (7,1) -- (7,-4);
\draw [-] (-1,0) -- (-1,-8);

\draw [-,thick] (0,2) -- (0,-8);
\draw [-,thick] (-2,0) -- (8,0);

\draw[-,thick] (-2,2) -- (8,2) -- (8,-8) -- (-2,-8) -- cycle;
\end{tikzpicture}
}
\end{center}

\begin{center}
\begin{tikzpicture}[scale=0.5]
\shadedraw[left color=yellow,right color = red] (3,0) rectangle (21,0.5);
\node at (3,-0.4) [] {\scriptsize $10^{-2}\, s$};
\node at (6,-0.4) [] {\scriptsize $10^{-1}\, s$};
\node at (9,-0.4) [] {\scriptsize $1\, s$};
\node at (12,-0.4) [] {\scriptsize $10\, s$};
\node at (15,-0.4) [] {\scriptsize $10^{2}\, s$};
\node at (18,-0.4) [] {\scriptsize $10^{3}\, s$};
\node at (21,-0.4) [] {\scriptsize $10^{4}\, s$};

\node at (12,0.9) [] {\scriptsize CPU-TIME};

\draw [-] (6,0) -- (6,0.5);
\draw [-] (9,0) -- (9,0.5);
\draw [-] (12,0) -- (12,0.5);
\draw [-] (15,0) -- (15,0.5);
\draw [-] (18,0) -- (18,0.5);
\end{tikzpicture}
\end{center}
\end{table}

Table \ref{tab:2345} suggests that for the strategies based on the \lq\lq multidegree\rq\rq\ and \lq\lq random data\rq\rq\ methods the dimension $n$ of the projective space affects the computational complexity slightly more than the degree $d$ of the Fermat hypersurface. The effect on the complexity of the number of variables and of the degree totally changes when using the \textbf{Strategy 5} or \textbf{Strategy 6} which rely on Theorem \ref{th:partitioning}. Table \ref{tab:2345} shows that the complexity of the \lq\lq partitioning\rq\rq\ method is heavily affected by the degree $d$ and mildly affected by the dimension of the ambient space $n$.
This behavior occurs because the effective computations in these strategies are performed on a number of variables bounded by $\min \{d,n+1\}$ and so eventually smaller than $n$. More precisely, assume that $\mathsf{a} \in \mathcal{P}_{n+1,d}$ and $\mathsf{b} \in \mathcal{P}_{n'+1,d}$ are two partitions of $n+1$ and $n'+1$ with same length $s$. The ideals $I_{\mathsf{a}}^d \setminus \mathcal{H}$ and $I_{\mathsf{b}}^d \setminus\mathcal{H}$ defining the critical points are both contained in the polynomial ring $\K[z_1,\ldots,z_s]$ and they are in fact almost the same ideal: they only differ by one generator ($a_1z_1^d+\cdots+a_sz_s^d$ instead of $b_1z_1^d+\cdots+b_sz_s^d$) and one different saturation ($a_1z_1+\cdots+a_sz_s$ instead of $b_1z_1+\cdots+b_sz_s$). For such ideals the complexity of degree computation does not change substantially. Hence, for a fixed $d$ and for $n \geqslant d-1$, the running time of ML degree algorithms grows slowly as $n$ increases, depending only on the number of partitions in $\mathcal{P}_{n+1,d}$. Thus, we need to repeat more times computations for ideals whose complexity is basically fixed.

\medskip

Finally, Table \ref{tab:MLdegree} lists the ML degree of the Fermat hypersurface $F_{n,d}$, for several values of $n$ and $d$. 

\begin{table}[!ht]
\begin{center}
\caption{\small\label{tab:MLdegree}  The ML degree of several Fermat hypersurfaces. The empty entries of the table correspond to ML degrees whose computation following the best strategy could not be completed within 24 hours of cpu-time.}

\medskip

\begin{tikzpicture}[scale=1.05]
\draw [-,very thick] (-2,0.55) -- (-2,-7.25) -- (11.6,-7.25) -- (11.6,0.55) -- cycle;
\draw [-,thick] (0,0.55) -- (0,-7.25);
\draw [-,thick] (-2,-0.25) -- (11.6,-0.25);


\draw [very thin,black!75,-] (0,0.15) -- (11.6,0.15);
\draw [very thin,black!75,-] (-1,-0.25) -- (-1,-7.25);

\draw [very thin,black!75,-] (-1,-0.75) -- (11.6,-0.75);
\draw [very thin,black!75,-] (-1,-1.25) -- (11.6,-1.25);
\draw [very thin,black!75,-] (-1,-1.75) -- (11.6,-1.75);
\draw [very thin,black!75,-] (-1,-2.25) -- (11.6,-2.25);
\draw [very thin,black!75,-] (-1,-2.75) -- (11.6,-2.75);
\draw [very thin,black!75,-] (-1,-3.25) -- (11.6,-3.25);
\draw [very thin,black!75,-] (-1,-3.75) -- (11.6,-3.75);
\draw [very thin,black!75,-] (-1,-4.25) -- (11.6,-4.25);
\draw [very thin,black!75,-] (-1,-4.75) -- (11.6,-4.75);
\draw [very thin,black!75,-] (-1,-5.25) -- (11.6,-5.25);
\draw [very thin,black!75,-] (-1,-5.75) -- (11.6,-5.75);
\draw [very thin,black!75,-] (-1,-6.25) -- (11.6,-6.25);
\draw [very thin,black!75,-] (-1,-6.75) -- (11.6,-6.75);
\draw [very thin,black!75,-] (-1,-7.25) -- (11.6,-7.25);

\node at (-1,0.07) [] {\small $\MLdeg F_{n,d}$};

\node at (-1.5,-3.75) [] {\footnotesize $n$};
\node at (-0.5,-0.5) [] {\footnotesize $2$};
\node at (-0.5,-1) [] {\footnotesize $3$};
\node at (-0.5,-1.5) [] {\footnotesize $4$};
\node at (-0.5,-2) [] {\footnotesize $5$};
\node at (-0.5,-2.5) [] {\footnotesize $6$};
\node at (-0.5,-3) [] {\footnotesize $7$};
\node at (-0.5,-3.5) [] {\footnotesize $8$};
\node at (-0.5,-4) [] {\footnotesize $9$};
\node at (-0.5,-4.5) [] {\footnotesize $10$};
\node at (-0.5,-5) [] {\footnotesize $11$};
\node at (-0.5,-5.5) [] {\footnotesize $12$};
\node at (-0.5,-6) [] {\footnotesize $13$};
\node at (-0.5,-6.5) [] {\footnotesize $14$};
\node at (-0.5,-7) [] {\footnotesize $15$};

\node at (5.8,0.35) [] {\footnotesize $d$};

\node at (0.4,-0.05) [] {\footnotesize $2$};
\node at (0.4,-0.5) [] {\tiny 6};
\node at (0.4,-1) [] {\tiny 14};
\node at (0.4,-1.5) [] {\tiny 30};
\node at (0.4,-2) [] {\tiny 62};
\node at (0.4,-2.5) [] {\tiny 126};
\node at (0.4,-3) [] {\tiny 254};
\node at (0.4,-3.5) [] {\tiny 510};
\node at (0.4,-4) [] {\tiny 1022};
\node at (0.4,-4.5) [] {\tiny 2046};
\node at (0.4,-5) [] {\tiny 4094};
\node at (0.4,-5.5) [] {\tiny 8190};
\node at (0.4,-6) [] {\tiny 16382};
\node at (0.4,-6.5) [] {\tiny 32766};
\node at (0.4,-7) [] {\tiny 65534};
\draw [very thin,black!75,-] (0.8,0.15) -- (0.8,-7.25);

\node at (1.4,-0.05) [] {\footnotesize $3$};
\node at (1.4,-0.5) [] {\tiny 9};
\node at (1.4,-1) [] {\tiny 30};
\node at (1.4,-1.5) [] {\tiny 95};
\node at (1.4,-2) [] {\tiny 293};
\node at (1.4,-2.5) [] {\tiny 896};
\node at (1.4,-3) [] {\tiny 2726};
\node at (1.4,-3.5) [] {\tiny 6813};
\node at (1.4,-4) [] {\tiny 25047};
\node at (1.4,-4.5) [] {\tiny 75746};
\node at (1.4,-5) [] {\tiny 228825};
\node at (1.4,-5.5) [] {\tiny 690690};
\node at (1.4,-6) [] {\tiny 2083370};
\node at (1.4,-6.5) [] {\tiny 6280649};
\node at (1.4,-7) [] {\tiny 18925046};
\draw [very thin,black!75,-] (2,0.15) -- (2,-7.25);

\node at (2.75,-0.05) [] {\footnotesize $4$};
\node at (2.75,-0.5) [] {\tiny 18};
\node at (2.75,-1) [] {\tiny 76};
\node at (2.75,-1.5) [] {\tiny 320};
\node at (2.75,-2) [] {\tiny 1294};
\node at (2.75,-2.5) [] {\tiny 5180};
\node at (2.75,-3) [] {\tiny 20892};
\node at (2.75,-3.5) [] {\tiny 84132};
\node at (2.75,-4) [] {\tiny 337384};
\node at (2.75,-4.5) [] {\tiny 1353110};
\node at (2.75,-5) [] {\tiny 5429494};
\node at (2.75,-5.5) [] {\tiny 21767018};
\node at (2.75,-6) [] {\tiny 87215496};
\node at (2.75,-6.5) [] {\tiny 349452578};
\node at (2.75,-7) [] {\tiny 1397573292};
\draw [very thin,black!75,-] (3.5,0.15) -- (3.5,-7.25);

\node at (4.35,-0.05) [] {\footnotesize $5$};
\node at (4.35,-0.5) [] {\tiny 27};
\node at (4.35,-1) [] {\tiny 140};
\node at (4.35,-1.5) [] {\tiny 725};
\node at (4.35,-2) [] {\tiny 3655};
\node at (4.35,-2.5) [] {\tiny 18494};
\node at (4.35,-3) [] {\tiny 92972};
\node at (4.35,-3.5) [] {\tiny 467685};
\node at (4.35,-4) [] {\tiny 2347469};
\node at (4.35,-4.5) [] {\tiny 11781044};
\node at (4.35,-5) [] {\tiny 59070599};
\node at (4.35,-5.5) [] {\tiny 296105784};
\node at (4.35,-6) [] {\tiny 1483630894};
\node at (4.35,-6.5) [] {\tiny 7432036277};
\node at (4.35,-7) [] {\tiny 37220018572};
\draw [very thin,black!75,-] (5.2,0.15) -- (5.2,-7.25);

\node at (6.1,-0.05) [] {\footnotesize $6$};
\node at (6.1,-0.5) [] {\tiny 42};
\node at (6.1,-1) [] {\tiny 258};
\node at (6.1,-1.5) [] {\tiny 1530};
\node at (6.1,-2) [] {\tiny 9186};
\node at (6.1,-2.5) [] {\tiny 55482};
\node at (6.1,-3) [] {\tiny 334578};
\node at (6.1,-3.5) [] {\tiny 2012514};
\node at (6.1,-4) [] {\tiny 12064506};
\node at (6.1,-4.5) [] {\tiny 72298842};
\node at (6.1,-5) [] {\tiny 433840578};
\node at (6.1,-5.5) [] {\tiny 2605621434};
\node at (6.1,-6) [] {\tiny 15650082090};
\node at (6.1,-6.5) [] {\tiny 93935183202};
\node at (6.1,-7) [] {\tiny 563502117618};
\draw [very thin,black!75,-] (7,0.15) -- (7,-7.25);

\node at (8,-0.05) [] {\footnotesize $7$};
\node at (8,-0.5) [] {\tiny 51};
\node at (8,-1) [] {\tiny 370};
\node at (8,-1.5) [] {\tiny 2635};
\node at (8,-2) [] {\tiny 18627};
\node at (8,-2.5) [] {\tiny 131320};
\node at (8,-3) [] {\tiny 924154};
\node at (8,-3.5) [] {\tiny 6496251};
\node at (8,-4) [] {\tiny 45627451};
\node at (8,-4.5) [] {\tiny 320280400};
\node at (8,-5) [] {\tiny 2247181471};
\node at (8,-5.5) [] {\tiny 15761369624};
\node at (8,-6) [] {\tiny 110517144758};
\node at (8,-6.5) [] {\tiny 774762908611};
\node at (8,-7) [] {\tiny 5430367540394};
\draw [very thin,black!75,-] (9,0.15) -- (9,-7.25);

\node at (9.5,-0.05) [] {\footnotesize $8$};
\node at (9.5,-0.5) [] {\tiny 72};
\node at (9.5,-1) [] {\tiny 584};
\node at (9.5,-1.5) [] {\tiny 4680};
\node at (9.5,-2) [] {\tiny 37448};
\node at (9.5,-2.5) [] {\tiny 298872};
\node at (9.5,-3) [] {\tiny };
\node at (9.5,-3.5) [] {\tiny };
\node at (9.5,-4) [] {\tiny };
\node at (9.5,-4.5) [] {\tiny };
\node at (9.5,-5) [] {\tiny };
\node at (9.5,-5.5) [] {\tiny };
\node at (9.5,-6) [] {\tiny };
\node at (9.5,-6.5) [] {\tiny };
\node at (9.5,-7) [] {\tiny };
\draw [very thin,black!75,-] (10,0.15) -- (10,-7.25);

\node at (10.4,-0.05) [] {\footnotesize $9$};
\node at (10.4,-0.5) [] {\tiny 87};
\node at (10.4,-1) [] {\tiny 792};
\node at (10.4,-1.5) [] {\tiny 7265};
\node at (10.4,-2) [] {\tiny };
\node at (10.4,-2.5) [] {\tiny };
\node at (10.4,-3) [] {\tiny };
\node at (10.4,-3.5) [] {\tiny };
\node at (10.4,-4) [] {\tiny };
\node at (10.4,-4.5) [] {\tiny };
\node at (10.4,-5) [] {\tiny };
\node at (10.4,-5.5) [] {\tiny };
\node at (10.4,-6) [] {\tiny };
\node at (10.4,-6.5) [] {\tiny };
\node at (10.4,-7) [] {\tiny };
\draw [very thin,black!75,-] (10.8,0.15) -- (10.8,-7.25);

\node at (11.2,-0.05) [] {\footnotesize $10$};
\node at (11.2,-0.5) [] {\tiny 108};
\node at (11.2,-1) [] {\tiny 1102};
\node at (11.2,-1.5) [] {\tiny 11090};
\node at (11.2,-2) [] {\tiny };
\node at (11.2,-2.5) [] {\tiny };
\node at (11.2,-3) [] {\tiny };
\node at (11.2,-3.5) [] {\tiny };
\node at (11.2,-4) [] {\tiny };
\node at (11.2,-4.5) [] {\tiny };
\node at (11.2,-5) [] {\tiny };
\node at (11.2,-5.5) [] {\tiny };
\node at (11.2,-6) [] {\tiny };
\node at (11.2,-6.5) [] {\tiny };
\node at (11.2,-7) [] {\tiny };

\end{tikzpicture}
\end{center}
\end{table}

\paragraph{\bf Acknowledgments.} The content of \cite{HuhSturmfels} has been the topic of a series of lectures taught by Bernd Sturmfels at the summer school \emph{Combinatorial algebraic geometry} in Levico Terme in June 2013. This paper arises from an exercise he suggested during a training session. We thank Bernd Sturmfels for having pushed and encouraged us to go into depth on the matter. We also thank the referees for having pointed out some inaccuracies and for the useful comments and suggestions.


\providecommand{\bysame}{\leavevmode\hbox to3em{\hrulefill}\thinspace}
\providecommand{\MR}{\relax\ifhmode\unskip\space\fi MR }
\providecommand{\MRhref}[2]{%
  \href{http://www.ams.org/mathscinet-getitem?mr=#1}{#2}
}
\providecommand{\href}[2]{#2}

\end{document}